\documentclass{amsart}
\usepackage{amssymb}
\usepackage{amscd}
\usepackage{amsmath}
\usepackage{amsrefs}
\usepackage{xypic}

\newcommand{\mydet}{\text{Det}}

\newtheorem{theorem}{Theorem}
\newtheorem{prop}{Proposition}
\newtheorem{lemma}{Lemma}
\newtheorem{corollary}{Corollary}

\newtheorem{definition}{Definition}

\newtheorem{remark}{Remark}

\newcommand{\bn}{\mathbb N}
\newcommand{\bq}{\mathbb Q}
\newcommand{\bz}{\mathbb Z}

\newcommand{\bof}{\mathbb F}

\newcommand{\qbar}{\bar{\mathbb Q}}

\newcommand{\hbq}{\hat{\bq}^{ur}}
\newcommand{\vw}{v_\varpi}

\newcommand{\co}{\mathcal O}

\newcommand{\p}{\mathfrak p}
\newcommand{\beq}{\begin{equation}}
\newcommand{\eeq}{\end{equation}}

\newcommand{\mnu}{\mu}

\DeclareMathOperator{\Ind}{Ind}

\DeclareMathOperator{\cyclo}{cyclo}
\DeclareMathOperator{\Exp}{Exp}

\DeclareMathOperator{\im}{im} 
\DeclareMathOperator{\Frob}{Frob} 
\DeclareMathOperator{\Hom}{Hom} 

\DeclareMathOperator{\Tr}{Tr}
\DeclareMathOperator{\Gal}{Gal}\DeclareMathOperator{\Aut}{Aut}
\DeclareMathOperator{\per}{per}

\DeclareMathOperator{\Det}{Det}
\DeclareMathOperator{\pr}{pr}

\def\rightiso{{~ \overset\sim\to ~}}
\DeclareMathOperator{\free}{{tf}}

\usepackage{enumerate}
\newcommand{\N}{{\mathbb N}}
\newcommand{\NN}{{\mathcal N}}
\newcommand{\OO}{{\mathcal O}}
\DeclareMathOperator{\tors}{{tors}}
\usepackage{stmaryrd}

\title[On the local Tamagawa number conjecture for Tate motives]
{On the local Tamagawa number conjecture for Tate motives over tamely ramified fields}

\author{J. Daigle and M. Flach}
\address{Department of Mathematics, Caltech, Pasadena CA 91125, USA}
\address{Department of Mathematics, Occidental College, Los Angeles CA 90041, USA}


\begin{document}
\begin{abstract} The local Tamagawa number conjecture, which was first formulated by Fontaine and Perrin-Riou, expresses the compatibility of the (global) Tamagawa number conjecture on motivic L-functions with the functional equation. The local conjecture was proven for Tate motives over finite unramified extensions $K/\bq_p$ by Bloch and Kato. We use the theory of $(\varphi,\Gamma)$-modules and a reciprocity law due to Cherbonnier and Colmez to provide a new proof in the case of unramified extensions, and to prove the conjecture for $\bq_p(2)$ over certain tamely ramified extensions.
\end{abstract}

\maketitle
\section{Introduction}

Let $K/\bq_p$ be a finite extension and $V$ a de Rham representation of $G_K:=\Gal(\bar{K}/K)$. The local Tamagawa number conjecture is a statement describing a certain $\bq_p$-basis of the determinant line $\det_{\bq_p} R\Gamma(K,V)$ of (continuous) local Galois cohomology up to units in $\bz_p^\times$. It was first formulated as conjecture $C_{EP}$ by Fontaine and Perrin-Riou \cite{fpr91}[4.5.4] and independently as the "local $\epsilon$-conjecture" by Kato \cite{kato932}[Conj. 1.8]. Both conjectures express compatibility of the (global) Tamagawa number conjecture on motivic L-functions with the functional equation. The fact that the local Tamagawa number conjecture is equivalent to this compatibility still constitutes its main interest. For example, the proof of the Tamagawa number conjecture for Dirichlet L-functions at integers $r\geq 2$ \cite{bufl06} uses the conjecture at $1-r$ and compatibility with the functional equation (no other more direct proof is known). In \cite{fk} Fukaya and Kato generalized \cite{kato932}[Conj. 1.8] to de Rham representations with coefficients in a possibly non-commutative $\bq_p$-algebra, and in fact to arbitrary $p$-adic families of local Galois representations.

In this paper we shall only consider Tate motives $V=\bq_p(r)$ with $r\geq 2$ (for the case $r=1$ see \cite{bleycobbe16}, \cite{breuning04}). If $K/\bq_p$ is {\em unramified} the local Tamagawa number conjecture for $\bq_p(r)$ was first proven by Bloch and Kato in their seminal paper \cite{bk88} on the global Tamagawa number conjecture, and has since been reproven by a number of authors (e.g. \cite{pr94}, \cite{bb05}). These later proofs also cover the case where $K/\bq_p$ is a cyclotomic extension, or more generally where $V$ is an abelian de Rham representations of $\Gal(\qbar_p/\bq_p)$ \cite{kato932}[Thm 4.1], \cite{venj13}. All proofs have two main ingredients: Iwasawa theory and a "reciprocity law". The latter is an explicit description of the exponential or dual exponential map for the deRham representation $V$, which however very often only holds in restricted situations (e.g. $V$ ordinary or absolutely crystalline).  The aim of this paper is to explore the application of the very general reciprocity law of Cherbonnier and Colmez \cite{chercol99}, which holds for arbitrary de Rham representations, to the local Tamagawa number conjecture for Tate motives.

In section \ref{conjectur} we shall give a first somewhat explicit statement (Prop. \ref{prop:cep}) which is equivalent to the local Tamagawa conjecture for $\bq_p(r)$ over an arbitrary Galois extension $K/\bq_p$. We shall in fact work with the refined equivariant conjecture over the group ring $\bz_p[\Gal(K/\bq_p)]$, following Fukaya and Kato \cite{fk}. In section \ref{sec:tame} we focus on the case where $p\nmid [K:\bq_p]$. In section \ref{sec:chercol} we state the reciprocity law of Cherbonnier and Colmez in the case of Tate motives. In section \ref{sec:unramified} we show that it also can be used to give a proof of the unramified case (which however has many common ingredients with the existing proofs). Finally, in section \ref{sec:tame2} we formulate our main result, Prop. \ref{explicit2}, which is a fairly explicit statement equivalent to the equivariant local Tamagawa number conjecture for $\bq_p(r)$ over $K/\bq_p$ with $p\nmid [K:\bq_p]$. We show that it can be used to prove some new cases, more specifically we have

\begin{prop} Assume $K/\bq_p$ is Galois of degree prime to $p$ and with ramification degree $e<p/4$. Then the equivariant local Tamagawa number conjecture holds for $V=\bq_p(2)$.
\end{prop}

The only cases where the conjecture for tamely ramified fields was known previously are cyclotomic fields, i.e. where $e\mid p-1$, and in this case one can allow arbitrary $r$  \cite{pr94}, \cite{bb05}. We think that many more cases can be proven with Prop. \ref{explicit2} and hope to come back to this in a subsequent article.
\bigskip

{\em Acknowledgements:} We would like to thank the referee for a very careful reading of the manuscript which helped to improve our exposition a lot.

\section{The conjecture}\label{conjectur}

Throughout this paper $p$ denotes an odd prime. Let $K/\bq_p$ be an arbitrary finite Galois extension with group $G$ and $r\geq 2$. In this section we shall explicate the consequences of the local Tamagawa number conjecture of Fukaya and Kato \cite{fk}[Conj. 3.4.3] for the triple $$(\Lambda,T,\zeta)=(\bz_p[G],\Ind^{G_{\bq_p}}_{G_K}\bz_p(1-r),\zeta).$$ Here $\zeta=(\zeta_{p^n})_n\in\Gamma(\qbar_p,\bz_p(1))$ is a compatible system of $p^n$-th roots of unity which we fix throughout this paper. The conjectures for a triple $(\Lambda,T,\zeta)$ and its dual $(\Lambda^{op},T^*(1),\zeta)$ are equivalent. We find it advantageous to work with $\bq_p(1-r)$ rather than $\bq_p(r)$ as in \cite{bk88} since we are employing the Cherbonnier-Colmez reciprocity law \cite{chercol99} which describes the dual exponential map.

In order to give an idea what the conjecture is about consider the Bloch-Kato exponential map \cite{bk88}
$$ \exp: K\xrightarrow{\sim}H^1(K,\bq_p(r)).$$
In a first approximation one may say that the local Tamagawa number conjecture describes the relation between the two $\bz_p$-lattices $\exp(\co_K)$ and
$\im(H^1(K,\bz_p(r)))$ inside $H^1(K,\bq_p(r))$. Rather than giving a complete description of the relative position of these two lattices, the conjecture only specifies their relative volume, that is the class in $\bq_p^\times/\bz_p^\times$ which multiplies $\mydet_{\bz_p}\exp(\co_K)$
to $\mydet_{\bz_p}(\im(H^1(K,\bz_p(r))))$ inside the $\bq_p$-line $\mydet_{\bq_p}H^1(K,\bq_p(r))$. The equivariant form of the conjecture is a finer statement which arises by replacing determinants over $\bz_p$ by determinants over $\bz_p[G]$. If $G$ is abelian and $\im(H^1(K,\bz_p(r)))$ is projective over $\bz_p[G]$, the conjecture thereby does specify the relative position of the two lattices in view of the fact that  $H^1(K,\bq_p(r))$ is free of rank one over $\bq_p[G]$ and so coincides with its determinant. If $G$ is non-abelian, even though $H^1(K,\bq_p(r))$ remains free of rank one over $\bq_p[G]$, the conjecture is an identity in the algebraic K-group $K_1(\bq^{ur}_p[G]))/K_1(\bz^{ur}_p[G]))$ and is again quite a bit weaker than a full determination of the relative position of the two lattices.

Determinants in the sense of \cite{deligne87} (see also \cite{fk}[1.2]) are only defined for modules of finite projective dimension, or more generally perfect complexes, and so
the first step is to replace the $\bz_p$-lattice $\im(H^1(K,\bz_p(r)))$ by the entire perfect complex $R\Gamma(K,\bz_p(r))$. There still is an isomorphism
 \begin{equation}R\Gamma(K,\bz_p(r))\otimes_{\bz_p}\bq_p\cong R\Gamma(K,\bq_p(r))\cong H^1(K,\bq_p(r))[-1]\label{e1}\end{equation}
since the groups $H^1(K,\bz_p(r))_{tor}$ and $H^2(K,\bz_p(r))$ are finite. If $K/\bq_p$ is Galois with group $G$ then $R\Gamma(K,\bz_p(r))$ is always a perfect complex of $\bz_p[G]$-modules whereas $\im(H^1(K,\bz_p(r)))$ or $\co_K$ need no longer have finite projective dimension over $\bz_p[G]$. A further simplification occurs if one does not try to compare $R\Gamma(K,\bz_p(r))$ to $\exp(\co_K)$ directly. Instead one uses the "period isomorphism"
$$ \per: \qbar_p\otimes_{\bq_p}K\cong \qbar_p\otimes_{\bq_p}\left(\Ind^{G_{\bq_p}}_{G_K}\bq_p\right)\cong\qbar_p[G]$$
and tries to compare $\mydet_{\bz_p}R\Gamma(K,\bz_p(r))$ to a suitable lattice in this last space. The left-$\bz_p[G]$-module $\Ind^{G_{\bq_p}}_{G_K}\bz_p$ is always free  of rank one whereas $\co_K$ need not be. After choosing an embedding $K\to\qbar_p$ one gets an isomorphism $\psi:G_{\bq_p}/G_K\cong G$ and an isomorphism
\begin{equation}\Ind^{G_{\bq_p}}_{G_K}\bz_p\cong\bz_p[G]\label{induced}\end{equation} so that the $\bz_p[G]$-linear left action of $\gamma\in G_{\bq_p}$  is given by
\begin{equation} \bz_p[G]\ni x\mapsto x\psi(\gamma^{-1}).\label{ra}\end{equation}The period isomorphism is then given for $x\in K$ by
$$ \per(x):=\per(1\otimes x)=\sum_{g\in G} g(x)\cdot g^{-1}\in \qbar_p[G].$$
The dual of $\exp$ identifies with the dual exponential map
$$ \exp^*_{\bq_p(r)}: H^1(K,\bq_p(1-r))\to K$$
by local Tate duality and the trace pairing on $K$.
Let $\beta\in H^1(K,\bz_p(1-r))$ be an element spanning a free $\bz_p[G]$-submodule and let $C_\beta$ be the mapping cone of the ensuing map of perfect complexes of $\bz_p[G]$-modules $$(\bz_p[G]\cdot\beta)[-1]\to H^1(K,\bz_p(1-r))[-1]\to R\Gamma(K,\bz_p(1-r)).$$ Then $C_\beta$ is a perfect complex of $\bz_p[G]$-modules with finite cohomology groups, i.e. such that $C_\beta\otimes_{\bz_p}\bq_p$ is acyclic. It therefore represents a class $[C_\beta]$ in the relative $K$-group $K_0(\bz_p[G],\bq_p)$ for which one has an exact sequence
$$ K_1(\bz_p[G])\to K_1(\bq_p[G])\to K_0(\bz_p[G],\bq_p) \to 0.$$
Hence we may also view $[C_\beta]$ as an element in $K_1(\bq_p[G])/\im(K_1(\bz_p[G]))$.
Extending scalars to $\qbar_p$ we get an isomorphism of free rank one $\qbar_p[G]$-modules
$$ H^1(K,\bq_p(1-r))\otimes_{\bq_p}\qbar_p\xrightarrow{\exp^*\otimes\qbar_p}K\otimes_{\bq_p}\qbar_p\xrightarrow{\per}\qbar_p[G]$$
sending the $\qbar_p[G]$-basis $\beta$ to a unit  $\per(\exp^*(\beta))\in \qbar_p[G]^\times$. As such it has a class $$[\per(\exp^*(\beta))] \in K_1(\qbar_p[G])$$ via the natural projection map $\qbar_p[G]^\times\to K_1(\qbar_p[G])$ (recall that for any ring $R$ we have maps $R^\times\to GL(R)\to GL(R)^{ab}=:K_1(R)$). In section \ref{eps} below we shall define an $\epsilon$-factor $\epsilon(K/\bq_p,1-r)\in K_1(\qbar_p[G])$ so that
$$\epsilon(K/\bq_p,1-r)\cdot[\per(\exp^*(\beta))] \in K_1(\bq^{ur}_p[G]).$$
Let $F\subseteq K$ denote the maximal unramified subfield, $\Sigma=\Gal(F/\bq_p)$ and $\sigma\in\Sigma$ the (arithmetic) Frobenius automorphism. Then $\bq_p[\Sigma]$ is canonically a direct factor of $\bq_p[G]$ and $\bq_p[\Sigma]^\times\cong K_1(\bq_p[\Sigma])$ a direct factor of $K_1(\bq_p[G])$. For $\alpha\in\bq_p[\Sigma]^\times$ we denote by $[\alpha]_F$ its class in $K_1(\bq_p[G])$.
Finally, note that if $R$ is a $\bq$-algebra then any nonzero rational number $n$ has a class $[n]\in K_1(R)$ via $\bq^\times\to R^\times\to K_1(R)$.
\bigskip

Then one has
\begin{prop} Let $K/\bq_p$ be Galois with group $G$ and $r\geq 2$. The local Tamagawa number conjecture for the triple $$(\Lambda,T,\zeta)=(\bz_p[G],\Ind^{G_{\bq_p}}_{G_K}\bz_p(1-r),\zeta).$$ is equivalent to the identity
\begin{equation}[(r-1)!]\cdot\epsilon(K/\bq_p,1-r)\cdot[\per(\exp^*(\beta))]\cdot [C_\beta]^{-1}\cdot\left[\frac{1-p^{r-1}\sigma}{1-p^{-r}\sigma^{-1}}\right]_F=1\label{cep}\end{equation}
in the group $K_1(\bq^{ur}_p[G]))/\im K_1(\bz^{ur}_p[G]))$.
\label{prop:cep}\end{prop}

Before we begin the proof of the proposition we explain what we mean by the local Tamagawa number conjecture for $(\bz_p[G],\Ind^{G_{\bq_p}}_{G_K}\bz_p(1-r),\zeta)$. The local Tamagawa number conjecture \cite{fk}[Conj. 3.4.3] claims the existence of $\epsilon$-isomorphisms $\epsilon_{\Lambda,\zeta}(T)$ for all triples $(\Lambda,T,\zeta)$ where $\Lambda$ is a semilocal pro-$p$ ring satisfying a certain finiteness condition \cite{fk}[1.4.1], $T$ a finitely generated projective $\Lambda$-module with continuous $G_{\bq_p}$-action and $\zeta$ a basis of $\Gamma(\qbar_p,\bz_p(1))$, such that certain functorial properties hold. One of these properties \cite{fk}[Conj. 3.4.3 (v)] says that if $L:=\Lambda\otimes_{\bz_p}\bq_p$ is a finite extension of $\bq_p$ and $V:=T\otimes_{\bz_p}\bq_p$ is a de Rham representation, then $$\tilde{L}\otimes_{\tilde{\Lambda}} \epsilon_{\Lambda,\zeta}(T)=\epsilon_{L,\zeta}(V)$$ where $\epsilon_{L,\zeta}(V)$
is the isomorphism in $C_{\tilde{L}}$ defined in \cite{fk}[3.3]. Here, for any ring $R$, $C_R$ is the Picard category constructed in \cite{fk}[1.2], equivalent to the category of virtual objects of \cite{deligne87}, $S\otimes_R-:C_R\to C_S$ is the Picard functor induced by a ring homomorphism $R\to S$ and $\tilde{R}=W(\bar{\bof}_p)\otimes_{\bz_p}R$ for any $\bz_p$-algebra $R$. The construction of $\epsilon_{L,\zeta}(V)$ involves certain isomorphisms and exact sequences which we recall in the proof below. If $A$ is a finite dimensional semisimple $\bq_p$-algebra and $V$ an $A$-linear de Rham representation those isomorphisms and exact sequences are in fact $A$-linear and therefore lead to an isomorphism $\epsilon_{A,\zeta}(V)$ in the category $C_{\tilde{A}}$. If $A:=\Lambda\otimes_{\bz_p}\bq_p$ is a semisimple $\bq_p$-algebra and $V:=T\otimes_{\bz_p}\bq_p$ is a de Rham representation, we say that the local Tamagawa number conjecture holds for the particular triple $(\Lambda,T,\zeta)$ if
$$\tilde{A}\otimes_{\tilde{\Lambda}} \epsilon_{\Lambda,\zeta}(T)=\epsilon_{A,\zeta}(V)$$
for some isomorphism $\epsilon_{\Lambda,\zeta}(T)$ in $C_{\tilde{\Lambda}}$.

\begin{proof} For a perfect complex of $\bq_p[G]$-modules $P$, we set $P^*=\Hom_{\bq_p[G]}(P,\bq_p[G])$ which is a perfect complex of $\bq_p[G]^{op}$-modules. Fix $r\geq 2$ and set \[V=\Ind^{G_{\bq_p}}_{G_K}\bq_p(1-r),  \text{  resp.  } V^*(1)=\Ind^{G_{\bq_p}}_{G_K}\bq_p(r)\] which is free of rank one over $\bq_p[G]$, resp. $\bq_p[G]^{op}$. We recall the ingredients of the isomorphism $\theta_{\bq_p[G]}(V)$ of \cite{fk}[3.3.2] (or rather of its generalization from field coefficients to semisimple coefficients). The element $\zeta$ determines an element $t=\log(\zeta)$ of $B_{dR}$. We have  \begin{align*} &D_{cris}(V)=F\cdot t^{r-1},\quad  D_{dR}(V)/D^0_{dR}(V)=0\\ &D_{cris}(V^*(1))=F\cdot t^{-r},\quad D_{dR}(V^*(1))/D^0_{dR}(V^*(1))=K,\end{align*}
\begin{align*} C_f(\bq_p,V)&: F\xrightarrow{1-p^{r-1}\sigma} F  \\
C_f(\bq_p,V^*(1))&: F\xrightarrow{(1-p^{-r}\sigma,\subseteq)} F\oplus K,\quad
\end{align*}
and commutative diagrams
\[\begin{CD}
\Det_{\bq_p[G]}(0) @>\eta(\bq_p,V)>> \Det_{\bq_p[G]}C_f(\bq_p,V)\cdot\Det_{\bq_p[G]}D_{dR}(V)/D^0_{dR}(V)\\
@AA[1-p^{r-1}\sigma]^{-1}_F A  @AAc A\\
\Det_{\bq_p[G]}(0) @>\eta'(\bq_p,V)>> \Det_{\bq_p[G]}(0)\cdot\Det_{\bq_p[G]}(0)^{-1}\cdot\Det_{\bq_p[G]}(0)
\end{CD}\]

\[\begin{CD}
\Det_{\bq_p[G]}(0) @>\eta(\bq_p,V^*(1))^{*,-1}>> \Det_{\bq_p[G]}C_f(\bq_p,V^*(1))^*\cdot(\Det_{\bq_p[G]}D_{dR}(V^*(1))/D^0_{dR}(V^*(1)))^*\\
@AA[1-p^{-r}\sigma^{-1}]_F A  @AAcA\\
\Det_{\bq_p[G]}(0) @>\eta'(\bq_p,V^*(1))^{*,-1}>> \Det_{\bq_p[G]}(0)\cdot\Det_{\bq_p[G]}(K^*)^{-1}
\cdot\Det_{\bq_p[G]}(K^*)
\end{CD}\]

\[\begin{CD}
\Det_{\bq_p[G]}C_f(\bq_p,V^*(1))^* @>\Det_{\bq_p[G]}\Psi_f(\bq_p,V^*(1))^{*,-1}>> \Det_{\bq_p[G]}\bigl(C(\bq_p,V)/C_f(\bq_p,V)\bigr)\\
@AAc A  @AA c A\\
\Det_{\bq_p[G]}(K^*)^{-1} @>\Psi'>> \Det_{\bq_p[G]}H^\bullet(\bq_p,V)
\end{CD}\]
where the vertical maps $c$ are induced by passage to cohomology. The morphism $\Psi'$ is ($\Det_{\bq_p[G]}^{-1}$ of) the inverse of the isomorphism
\[ H^1(\bq_p,V)\xrightarrow{T}H^1(\bq_p,V^*(1))^*\xrightarrow{\exp_{V^*(1)}^*}K^*\]
where $T$ is the local Tate duality isomorphism. For the isomorphism
\[ \theta_{\bq_p[G]}(V)=\eta(\bq_p,V)\cdot\left(\Det_{\bq_p[G]}\Psi_f(\bq_p,V^*(1))^{*,-1}\circ\eta(\bq_p,V^*(1))^{*,-1}\right)\]
we obtain a commutative diagram
\[\begin{CD}
\Det_{\bq_p[G]}(0) @>\theta_{\bq_p[G]}(V)>> \Det_{\bq_p[G]}C(\bq_p,V)\cdot\Det_{\bq_p[G]}D_{dR}(V)\\
@AA\left[\frac{1-p^{-r}\sigma^{-1}}{1-p^{r-1}\sigma}\right]_F A  @AAc A\\
\Det_{\bq_p[G]}(0) @>\theta'>> \Det_{\bq_p[G]}H^\bullet(\bq_p,V)\cdot \Det_{\bq_p[G]}(K)
\end{CD}\]
where $\theta'$ is induced by the dual exponential map
\[ H^1(\bq_p,V)\xrightarrow{\exp_{V^*(1)}^*}K.\]
The isomorphism $\Gamma_{\bq_p[G]}(V)\cdot\epsilon_{\bq_p[G],\zeta,dR}(V)$ of \cite{fk}[3.3.3] is the isomorphism
\[ [(-1)^{r-1}(r-1)!]\cdot\epsilon(K/\bq_p,1-r)\cdot \Det_{\qbar_p[G]}(\per) \]
and the isomorphism
\[ \epsilon_{\bq_p[G],\zeta}(V)=\Gamma_{\bq_p[G]}(V)\cdot\epsilon_{\bq_p[G],\zeta,dR}(V)\cdot\theta_{\bq_p[G]}(V)\]
fits into a commutative diagram
\[\begin{CD}
\Det_{\bq_p^{ur}[G]}(0) @>\epsilon_{\bq_p[G],\zeta}(V)>>\bq_p^{ur}[G]\underset{\bq_p[G]}{\otimes} \left(\Det_{\bq_p[G]}R\Gamma(K,\bq_p(1-r))\cdot\Det_{\bq_p[G]}(V)\right)\\
@AA\left[\frac{1-p^{-r}\sigma^{-1}}{1-p^{r-1}\sigma}\right]_F  A  @AAc A\\
\Det_{\bq_p^{ur}[G]}(0) @>\theta''>>\bq_p^{ur}[G]\underset{\bq_p[G]}{\otimes} \left( \Det_{\bq_p[G]}^{-1}H^1(K,\bq_p(1-r))\cdot \Det_{\bq_p[G]}(\bq_p[G])\right)
\end{CD}\]
where $$\theta''=[(-1)^{r-1}(r-1)!]\cdot\epsilon(K/\bq_p,1-r)\cdot \Det_{\qbar_p[G]}(\per)\cdot\theta'$$ and $c$ involves passage to cohomology as well as our identification $V\cong\bq_p[G]$ chosen above. Now passage to cohomology is also the scalar extension of the isomorphism
\[\Det_{\bz_p[G]}^{-1}(\bz_p[G]\cdot\beta)\cdot\Det_{\bz_p[G]}(C_\beta)\cong \Det_{\bz_p[G]}R\Gamma(K,\bz_p(1-r))\]
induced by the short exact sequence of perfect complexes of $\bz_p[G]$-modules
\[ 0\to R\Gamma(K,\bz_p(1-r))\to C_\beta\to \bz_p[G]\cdot\beta\to 0\]
combined with the acyclicity isomorphism
\[ \mathrm{can}:\Det_{\bq_p[G]}(0)\cong\Det_{\bq_p[G]}(C_{\beta,\bq_p}).\]
Since the class of $C_\beta$ in $K_0(\bz_p[G])$ vanishes we can choose an isomorphism
\[ a:\Det_{\bz_p[G]}(0)\cong\Det_{\bz_p[G]}(C_\beta)\]
which leads to another isomorphism
\[c':\Det_{\bz_p[G]}^{-1}(\bz_p[G]\cdot\beta)\cong \Det_{\bz_p[G]}R\Gamma(K,\bz_p(1-r))\]
defined over $\bz_p[G]$. Setting
\[\lambda:=(c_{\bq_p}')^{-1}c\in\Aut\left(\Det_{\bq_p[G]}^{-1}H^1(K,\bq_p(1-r))\right)=K_1(\bq_p[G])\]
we obtain a commutative diagram
\[\begin{CD}
\Det_{\bq_p^{ur}[G]}(0) @>\epsilon_{\bq_p[G],\zeta}(V)>>\bq_p^{ur}[G]\underset{\bq_p[G]}{\otimes} \left(\Det_{\bq_p[G]}R\Gamma(K,\bq_p(1-r))\cdot\Det_{\bq_p[G]}(V)\right)\\
@AA\left[\frac{1-p^{-r}\sigma^{-1}}{1-p^{r-1}\sigma}\right]_F  A  @AAc'_{\bq_p} A\\
\Det_{\bq_p^{ur}[G]}(0) @>\theta'''>>\bq_p^{ur}[G]\underset{\bq_p[G]}{\otimes} \left( \Det_{\bq_p[G]}^{-1}H^1(K,\bq_p(1-r))\cdot \Det_{\bq_p[G]}(\bq_p[G])\right)
\end{CD}\]
where
\[ \theta'''=\lambda\circ\theta''=\lambda\cdot [(-1)^{r-1}(r-1)!]\cdot\epsilon(K/\bq_p,1-r)\cdot \Det_{\qbar_p[G]}(\per)\cdot\theta'.\]
The local Tamagawa number conjecture claims that $\epsilon_{\bq_p[G],\zeta}(V)$ is induced by an isomorphism
\[\Det_{\bz_p^{ur}[G]}(0) \xrightarrow{\epsilon_{\bz_p[G],\zeta}(T)}\bz_p^{ur}[G]\underset{\bz_p[G]}{\otimes}
\left(\Det_{\bz_p[G]}R\Gamma(K,\bz_p(1-r))\cdot\Det_{\bz_p[G]}(T)\right)\]
and this will be the case if and only if
\[\theta^{\mathrm{iv}}:=\theta'''\cdot \left[\frac{1-p^{r-1}\sigma}{1-p^{-r}\sigma^{-1}}\right]_F \]
is induced by an isomorphism
\[\Det_{\bz_p^{ur}[G]}(0) \xrightarrow{\theta^{\mathrm{iv}}_{\bz_p[G]}}\bz_p^{ur}[G]\underset{\bz_p[G]}{\otimes} \left(\Det_{\bz_p[G]}^{-1}(\bz_p[G]\cdot\beta)\cdot\Det_{\bz_p[G]}(\bz_p[G])\right).\]
The isomorphism of $\qbar_p[G]$-modules
$$ \tau: H^1(K,\bq_p(1-r))\otimes_{\bq_p}\qbar_p\xrightarrow{\exp^*\otimes\qbar_p}K\otimes_{\bq_p}\qbar_p\xrightarrow{\per}
\qbar_p[G]\xrightarrow{\cdot \per(\exp^*(\beta))^{-1}}\qbar_p[G] $$
is clearly induced by an isomorphism of $\bz_p[G]$-modules
$$ \tau_{\bz_p[G]}: \bz_p[G]\cdot\beta \xrightarrow{\sim}\bz_p[G]$$
and we have
\[\theta^{\mathrm{iv}}=\left[\frac{1-p^{r-1}\sigma}{1-p^{-r}\sigma^{-1}}\right]_F\cdot\lambda\cdot [(-1)^{r-1}(r-1)!]\cdot\epsilon(K/\bq_p,1-r)\cdot[\per(\exp^*(\beta))]\cdot\Det_{\qbar_p[G]}(\tau).\]
Hence $\theta^{\mathrm{iv}}$ is induced by an isomorphism $\theta^{\mathrm{iv}}_{\bz_p[G]}$ if and only if
the class in $K_1(\bq_p^{ur}[G])$ of
\[\left[\frac{1-p^{r-1}\sigma}{1-p^{-r}\sigma^{-1}}\right]_F\cdot\lambda\circ [(-1)^{r-1}(r-1)!]\cdot\epsilon(K/\bq_p,1-r)\cdot [\per(\exp^*(\beta))]\]
lies in  $K_1(\bz_p^{ur}[G])$. Now note that $[(-1)]\in K_1(\bz)\subset K_1(\bz_p^{ur}[G])$ and that $\lambda=[C_\beta]^{-1}$, so we do indeed obtain identity (\ref{cep}). In order to see this last identity note that we have
\[ \lambda^{-1}
=a^{-1}\cdot\mathrm{can}\]
and that $a^{-1}\cdot\mathrm{can}\in K_1(\bq_p^{ur}[G])$ is a lift of $[C_\beta]\in K_0(\bz_p^{ur}[G],\bq_p)$ according to the conventions of \cite{fk}[1.3.8, Thm.1.3.15 (ii)].
\end{proof}

\subsection{Description of $K_1$} For any finite group $G$ we have the Wedderburn decomposition
\[\qbar_p[G]\cong \prod_{\chi\in\hat{G}}M_{d_\chi}(\qbar_p)\]
 where $\hat{G}$ is the set of irreducible $\qbar_p$-valued characters of $G$ and $d_\chi=\chi(1)$ is the degree of $\chi$. Hence a corresponding decomposition
\begin{equation} K_1(\qbar_p[G])\cong \prod_{\chi\in\hat{G}}K_1(M_{d_\chi}(\qbar_p))\cong \prod_{\chi\in\hat{G}}\qbar_p^\times\label{k1dec}\end{equation}
which allows one to think of $K_1(\qbar_p[G])$ as a collection of nonzero $p$-adic numbers indexed by $\hat{G}$.
Note here that for any ring $R$ one has $K_1(M_d(R))=K_1(R)$ and for a commutative semilocal ring $R$ one has $K_1(R)=R^\times$.

If $p\nmid|G|$ then all characters $\chi\in\hat{G}$ take values in $\bz_p^{ur}$, the Wedderburn decomposition is already defined over $\bz_p^{ur}$ and so is the decomposition of $K_1$. One has
\begin{equation} K_1(\bz_p^{ur}[G])\cong \prod_{\chi\in\hat{G}}K_1(M_{d_\chi}(\bz^{ur}_p))\cong \prod_{\chi\in\hat{G}}\bz_p^{ur,\times}\notag\end{equation}
and
\begin{equation} K_1(\bq_p^{ur}[G])/\im(K_1(\bz_p^{ur}[G]))\cong\prod_{\chi\in\hat{G}}\bq_p^{ur,\times}/\bz_p^{ur,\times}\cong
\prod_{\chi\in\hat{G}}p^\bz\label{k1urdec}\end{equation}
which allows one to think of elements in $K_1(\bq_p^{ur}[G])/\im(K_1(\bz_p^{ur}[G]))$ as a collection of integers ($p$-adic valuations) indexed by $\hat{G}$.

\subsection{Definition of the $\epsilon$-factor}\label{eps} If $L$ is a local field, $E$ an algebraically closed field of characteristic $0$ with the discrete topology, $\mu_L$ a Haar measure on the additive group of $L$ with values in $E$, $\psi_L:L\to E^\times$ a continuous character, the theory of Langlands-Deligne \cite{deligne73} associates to each continuous representation $r$ of the Weil group $W_L$ over $E$ an $\epsilon$-factor $\epsilon(r,\psi_L,\mu_L)\in E^\times$.

We shall take $E=\qbar_p$ and always fix $\mu_L$ and $\psi_L$ so that $\mu(\co_L)=1$ and $\psi_L=\psi_{\bq_p}\circ\Tr_{L/\bq_p}$ where $\psi_{\bq_p}(p^{-n})=\zeta_{p^n}$ for our fixed $\zeta=(\zeta_{p^n})_n\in\Gamma(\qbar_p,\bz_p(1))$. Setting $$\epsilon(r):=\epsilon(r,\psi_L,\mu_L)\in E^\times$$ and leaving the dependence on $\zeta$ implicit, we have the following properties (see also \cite{bb05} for a review, \cite{fk} only reviews the case $L=\bq_p$). Let $\pi$ be a uniformizer of $\co_L$, $\delta_L$ the exponent of the different of $L/\bq_p$ and $q=|\co_L/\pi|$.
\begin{itemize}
\item[a)] If $r:W_L\to E^\times$ is a homomorphism, set
$$r_\sharp:L^\times\xrightarrow{\mathrm{rec}}W_L^{ab}\xrightarrow{r}E^\times$$
where $\mathrm{rec}$ is normalized as in \cite{deligne73}[(2.3)] and sends a uniformizer to a {\em geometric} Frobenius automorphism in $W_L^{ab}$. Then we have
$$\epsilon(r)=\begin{cases} q^{\delta_L} & \text{if $c=0$}\\ q^{\delta_L}r_\sharp(\pi^{c+\delta_L})\tau(r_\sharp,\psi_\pi) & \text{if $c>0$} \end{cases}$$ where $c\in\bz$ is the conductor of $r$ and
\begin{equation}\tau(r_\sharp,\psi_\pi)=\sum_{u\in(\co_L/\pi^c)^\times}r_\sharp^{-1}(u)\psi_\pi(u)\label{gaussdef}\end{equation} is the Gauss sum associated to the restriction of $r_\sharp$ to $(\co_L/(\pi^c))^\times$ and the additive character
$$u\mapsto\psi_\pi(u):=\psi_K(\pi^{-\delta_L-c}u)$$ of $\co_L/(\pi^c)$.
\item[b)] If $L/K$ is unramified then $\epsilon(r)=\epsilon(\Ind_{W_L}^{W_K}r)$ for any representation $r$ of $W_L$.
\item[c)] If $r(\alpha)$ is the twist of $r$ with the unramified character with $\Frob_L$-eigenvalue $\alpha\in E^\times$, and $c(r)\in\bz$ is the conductor of $r$, then $$\epsilon(r(\alpha))=\alpha^{-c(r)-\dim_E(r)\delta_L}\epsilon(r).$$
    Here $\Frob_L$ denotes the usual (arithmetic) Frobenius automorphism.
\end{itemize}

For a potentially semistable representation $V$ of $G_{\bq_p}$ one first forms $D_{pst}(V)$, a finite dimensional $\hbq_p$-vector space of dimension $\dim_{\bq_p}V$ with an action of $G_{\bq_p}$, semilinear with respect to the natural action of $G_{\bq_p}$ on $\hbq_p$ and discrete on the inertia subgroup. Moreover, $D_{pst}(V)$ has a $\Frob$-semilinear automorphism $\varphi$. The associated linear representation $r_V$ of $W_{\bq_p}$ over $E=\hbq_p$ is the space $D_{pst}(V)$ with action
$$r_V(w)(d)=\iota(w)\varphi^{-\nu(w)}(d)$$
where $\iota:W_{\bq_p}\to G_{\bq_p}$ is the inclusion and $\nu(w)\in\bz$ is such that $\Frob^{\nu(w)}$ is the image of $w$ in $G_{\bof_p}$.

From now on we are interested in  $V=(\Ind^{G_{\bq_p}}_{G_K}\bq_p)(1-r)$. Here one has $$D_{pst}(V)=(\Ind^{G_{\bq_p}}_{G_K}\hbq_p)\cdot t^{r-1},\quad\quad r_V=(\Ind_{W_K}^{W_{\bq_p}}\hbq_p)(p^{1-r})$$ and we notice that $r_V$ is the scalar extension from $\bq_p^{ur}$ to $\hbq_p$ of the representation $(\Ind_{W_K}^{W_{\bq_p}}\bq^{ur}_p)(p^{1-r})$. So completion of $\bq_p^{ur}$ is not needed in this example.  Associated to $r_V\otimes_{\bq_p^{ur}}\qbar_p$ is an $\epsilon$-factor in $\epsilon(r_V)\in\qbar_p^\times=K_1(\qbar_p)$. However, as explained above before (\ref{ra}), $r_V$ carries a left action of $\bq^{ur}_p[G]$ commuting with the left  $W_{\bq_p}$-action, so we will actually be able to associate to $r_V\otimes_{\bq_p^{ur}}\qbar_p$ a refined $\epsilon$-factor
$$ \epsilon(K/\bq_p,1-r)\in K_1(\qbar_p[G]).$$

For each $\chi\in\hat{G}$ define a representation $r_\chi$ of $W_{\bq_p}$ over $E=\qbar_p$ by
\begin{equation}W_{\bq_p}\xrightarrow{\iota} G_{\bq_p}\xrightarrow{\psi}G\xrightarrow{\rho_\chi} GL_{d_\chi}(E)\label{rdef}\end{equation}
where $\rho_\chi:G\to GL_{d_\chi}(E)$ is a homomorphism realizing $\chi$.
Let $E^{d_\chi}$ be the space of row vectors on which $G$ acts on the right via $\rho_\chi$ and define another  representation of $W_{\bq_p}$ over $E=\qbar_p$
 $$ r_{V,\chi}=E^{d_\chi}\otimes_{\bq^{ur}_p[G]}r_V=E^{d_\chi}\otimes_{\bq^{ur}_p[G]}(\Ind_{W_K}^{W_{\bq_p}}\bq^{ur}_p)(p^{1-r})\cong E^{d_\chi}.$$
By (\ref{ra}) the left $W_{\bq_p}$-action on this last space is given by the contragredient ${^t}\rho_\chi(\psi(g))^{-1}$ of $r_\chi$, twisted by the unramified character with eigenvalue $p^{1-r}$. So we have $$r_{V,\chi}\cong r_{\bar{\chi}}(p^{1-r})$$ where $\bar{\chi}$ is the contragredient character of $\chi$. We view the collection
\begin{equation} \epsilon(K/\bq,1-r):=(\epsilon(r_{V,\chi}))_{\chi\in\hat{G}}=(\epsilon(r_{\bar{\chi}})p^{(r-1)c(r_{\bar{\chi}})})_{\chi\in\hat{G}}\label{epscompute}\end{equation}
as an element of $K_1(\qbar_p[G])$ in the description (\ref{k1dec}).

\section{The conjecture in the case $p\nmid|G|$}\label{sec:tame}

From now on and for most of the rest of the paper we assume that $p$ does not divide  $|G|=[K:\bq_p]$. In particular $K/\bq_p$ is tamely ramified with maximal unramified subfield $F$. Although our methods probably extend to an arbitrary tamely ramified extension $K/\bq_p$ (i.e. where $p$ is allowed to divide $[F:\bq_p]$) this would add an extra layer of notational complexity which we have preferred to avoid. The group $G=\Gal(K/\bq_p)$ is an extension of two cyclic groups
\begin{align*}\Sigma&:=\Gal(F/\bq_p)\cong \bz/f\bz\\ \Delta&:=\Gal(K/F)\cong \bz/e\bz\end{align*}
where the action of $\sigma\in\Sigma$ on $\Delta$ is given by $\delta\mapsto \delta^p$ and we have $e\mid p^f-1$. By Kummer theory $K=F\left(\root e\of p_0\right)$ where $p_0\in (F^\times/(F^\times)^e)^\Sigma$ has order $e$. We can and will assume that $p_0$ has $p$-adic valuation one, and in fact that $p_0=\lambda\cdot p$ with $\lambda\in\mu_F$. Writing $p_0=\lambda'\cdot p_0'$ with $p_0'\in\bq_p$ we see that $K$ is contained in $F'(\root e\of{p_0'})$ where $F':=F(\root e\of{\lambda'})$ is unramified over $\bq_p$ and $p_0'$ is any choice of element in $\mu_{\bq_p}\cdot p=\mu_{p-1}\cdot p$. Since for the purpose of proving the local Tamagawa number conjecture we can always enlarge $K$, we may and will assume that $$K=F\left(\root e\of {p_0}\right),\quad p_0\in \mu_{p-1}\cdot p\subseteq \bq_p.$$
We then have
\[G=\Gal(K/\bq_p)\cong \Sigma\ltimes\Delta\]
since $\Gal(K/\bq_p(\root e\of p_0))$ is a complement of $\Delta$. If $(e,p-1)=1$ then the fields $K=F(\root e\of p_0)$ for $p_0\in \mu_{p-1}\cdot p$ are all isomorphic; in fact any Galois extension $K/\bq_p$ with invariants $e$ and $f$ is then isomorphic to the field $F(\root e\of p)$.

The choice of $p_0$ (in fact just the valuation of $p_0$) determines a character
\begin{equation}\eta_0:\Delta\xrightarrow{\sim}\mu_e\subset F^\times\subset \bq_p^{ur,\times}\subset\qbar_p^\times\label{eta0def}\end{equation}
by the usual formula $\delta\left(\root e\of p_0\right)=\eta_0(\delta)\cdot\root e\of p_0$. Let
$$ \eta:\Delta\to F^\times$$
be any character of $\Delta$ and
$$\Sigma_\eta:=\{g\in\Sigma | \forall \delta\in\Delta\,\,\eta(g\delta g^{-1})=\eta(\delta)\}$$
the stabilizer of $\eta$. Then for any character $\eta':\Sigma_\eta\to\bq_p^{ur,\times}$ we obtain a character
$$ \eta'\eta:G_\eta:=\Sigma_\eta\ltimes\Delta\to\bq_p^{ur,\times}$$
and an induced character
$$ \chi:=\Ind_{G_\eta}^G(\eta'\eta)$$
of $G$. By \cite{langalg}[Exerc. XVIII.7]  all irreducible characters of $G$ are obtained by this construction, and in fact each $\chi\in\hat{G}$ is parametrized by a unique pair $([\eta],\eta')$ where $[\eta]$ denotes the $\Sigma$-orbit of $\eta$.
The degree of $\chi$ is given by
\begin{equation} d_\chi=\chi(1)=f_\eta:=[\Sigma:\Sigma_\eta]=[F_\eta:\bq_p]   \end{equation}
where $F_\eta\subseteq F$ is the fixed field of $\Sigma_\eta$.

We have
$$ r_\chi=\Ind_{W_{F_\eta}}^{W_{\bq_p}}(r_{\eta'\eta} )
$$
where $r_\chi$ (resp. $r_{\eta'\eta}$) is the representation of $W_{\bq_p}$ (resp. $W_{F_\eta}$) defined as in (\ref{rdef}). By \cite{serre95}[Ch. VI. Cor. to Prop.4] we have $$c(r_\chi)=f_\eta c(r_\eta)=\begin{cases} 0 & \eta=1\\ f_\eta &\eta\neq 1. \end{cases}$$
Using b), c)  and a) of section \ref{eps} we have
\begin{equation} \epsilon(r_\chi)=\epsilon(r_{\eta'\eta})=\begin{cases} 1 & \eta=1 \\ \epsilon(r_\eta)r_{\eta'}(\Frob_{F_\eta})^{-c(r_\eta)}=\eta(\mathrm{rec}(p))\tau(r_{\eta,\sharp},\psi_p)\eta'(\sigma^{f_\eta})^{-1} & \eta\neq 1\end{cases}.\label{epscomp}\end{equation}

\subsection{Gauss sums} If $k_\eta$ denotes the residue field of $F_\eta$, we have a canonical character
$$\omega:k_\eta^\times\xleftarrow{\sim}\mu_{p^{f_\eta}-1}\subseteq F_\eta^\times\subseteq K^\times \subseteq\qbar_p^\times$$
where the first arrow is reduction mod $p$. On the other hand we have our character
\begin{equation} r_{\eta,\sharp}:F_\eta^\times\xrightarrow{\text{rec}}W^{ab}_{F_\eta}\xrightarrow{\iota} G_{F_\eta}^{ab}\xrightarrow{\psi}G_\eta^{ab}\xrightarrow{\eta} \qbar_p^\times \notag\end{equation}
of order dividing $e$. So there exists a unique $m_\eta\in\bz/e\bz$ such that
\begin{equation}r_{\eta,\sharp}\vert_{\mu_{p^{f_\eta}-1}}=\omega^{m_\eta(p^{f_\eta}-1)/e} \label{metadef}\end{equation}
and formula (\ref{gaussdef}) gives $$\tau(r_{\eta,\sharp},\psi_p)=\tau(\omega^{-m_\eta(p^{f_\eta}-1)/e})$$ where $$ \tau(\omega^{-i}):=\sum_{a\in k_\eta^\times}\omega(a)^{-i}\zeta_p^{\Tr_{k_\eta/\mathbb F_p}(a)} $$
is a Gauss sum associated to the finite field $k_\eta$. The $p$-adic valuation of these sums is known:

\begin{lemma}  For $0\leq i\leq p^{f_\eta}-1$ let $i=i_0+pi_1+p^2i_2+\cdots+i_{f_\eta-1}p^{f_\eta-1}$ be the $p$-adic expansion with digits $0\leq i_j\leq p-1$. Then $$v_p(\tau(\omega^{-i}))=\frac{i_0+i_1+\cdots+i_{f_\eta-1}}{p-1}=\sum_{j=0}^{f_\eta-1}\left\langle\frac{ip^j}{p^{f_\eta}-1}\right\rangle$$
where $v_p:\qbar_p^\times\to\bq$ is the $p$-adic valuation on $\qbar_p$ normalized by $v_p(p)=1$ and $0\leq\langle x\rangle<1$ is the fractional part of the real number $x$.
\end{lemma}

\begin{proof} This is \cite{wash}[Prop. 6.13 and Lemma 6.14].\end{proof}

\begin{corollary} For all $\eta\in\hat{\Delta}$ we have
$$v_p(\tau(r_{\eta,\sharp},\psi_p))=\sum_{j=0}^{f_\eta-1}\left\langle\frac{m_\eta p^j}{e}\right\rangle.$$
\label{epsval}\end{corollary}

After this interlude on Gauss sums we now prove a statement about periods of certain specific elements in $K$ which will eliminate any further reference to $\epsilon$-factors in the proof of Conjecture \ref{cep}.

\begin{prop}  Let $K/\bq_p$ be Galois with group $G$ of order prime to $p$. Then any fractional $\co_K$-ideal is a free $\bz_p[G]$-module of rank $1$ and
\[(\epsilon(r_{\bar{\chi}}))_{\chi\in\hat{G}}\cdot [\per(b)]\in \im(K_1(\bz_p^{ur}[G])) \]
for any  $\bz_p[G]$-basis $b$ of the inverse different $\left(\root e\of p_0\right)^{-\delta_K}\co_K=\left(\root e\of p_0\right)^{-(e-1)}\co_K$.
\label{froeh}\end{prop}

\begin{proof} This is a classical result in Galois module theory which can be found in \cite{froehlich} but rather than trying to match our notation to that paper we go through the main computations again. In this proof $\sigma$ will temporarily denote a generic element of $\Sigma$ rather than the Frobenius.

The image of $[\per(b)]$ in the $\chi$-component of the decomposition (\ref{k1dec}) is the $d_\chi\times d_\chi$-determinant
\[  [\per(b)]_\chi:=\det\  \rho_\chi\left(\sum_{g\in G}g(b)\cdot g^{-1}\right)=\det \sum_{g\in G}g(b) \rho_\chi(g)^{-1}\in\qbar_p^\times.\]
This character function is traditionally called a resolvent.
With notations as above,  $\left(\root e\of p_0\right)^{-(e-1)}\co_K$ is a free $\bz_p[G_\eta]$-module with basis $\sigma(b)$ where $\sigma\in G_\eta \backslash G\cong \Sigma_\eta\backslash \Sigma$ runs through a set of right coset representatives. The image of this basis under the period map is
\[  \per(\sigma(b))=\sum_{g\in G} g\sigma(b)\cdot g^{-1}=\sum_{\tau\in\Sigma_\eta\backslash\Sigma}
\left(\sum_{g\in G_\eta} \tau^{-1} g\sigma(b)\cdot  g^{-1} \right)\tau  \]
and if $\chi=\Ind_{G_\eta}^G(\chi')$ is an induced character we have by \cite{froehlich}[(5.15)]
\[ \rho_\chi\left(\sum_{g\in G}g(b)\cdot g^{-1}\right)=  \left(\sum_{g\in G_\eta} \tau^{-1} g\sigma(b)\cdot  \rho_{\chi'}(g)^{-1} \right)_{\sigma,\tau}.\]
In our case of interest $\chi'=\eta'\eta$ is a one-dimensional character. Write $$b=\xi\cdot x$$ where $x$ is an $\co_F[\Delta]$-basis of $\left(\root e\of p_0\right)^{-(e-1)}\co_K$ fixed by $\Sigma$ and $\xi$ a $\bz_p[\Sigma]$-basis of $\co_F$. Then writing $g=\delta\sigma'$ with $\delta\in\Delta$ and $\sigma'\in \Sigma_\eta$ this matrix becomes
\[ \left(\sum_{\sigma'\in\Sigma_\eta} \tau^{-1}\sigma'\sigma(\xi) \eta'(\sigma')^{-1}\sum_{\delta\in\Delta} \tau^{-1} \delta(x)\cdot \eta(\delta)^{-1} \right)_{\sigma,\tau} \]
and its determinant is
\[\det\  \left(\sum_{\sigma'\in\Sigma_\eta} \tau^{-1}\sigma'\sigma(\xi) \eta'(\sigma')^{-1}\right)_{\sigma,\tau}\cdot\prod_{\tau\in\Sigma_\eta\backslash\Sigma}\ \sum_{\delta\in\Delta} \tau^{-1} \delta\tau(x)\cdot \eta(\delta)^{-1}.  \]
The first determinant is a group determinant \cite{wash}[Lemma 5.26] for the group $\Sigma_\eta\backslash \Sigma$ and equals
\[ \xi_{\eta'}:= \prod_{\kappa\in\widehat{\Sigma_\eta\backslash \Sigma} }\sum_{\sigma\in\Sigma_\eta\backslash \Sigma}\left(\sum_{\sigma'\in\Sigma_\eta} \sigma'\sigma(\xi) \eta'(\sigma')^{-1} \right)\kappa(\sigma)^{-1} =\prod_{\kappa}
\sum_{\sigma\in\Sigma} \sigma(\xi) \kappa(\sigma)^{-1} \]
where this last product is over all characters $\kappa$ of $\Sigma$ restricting to $\eta'$ on $\Sigma_\eta$.
The sum $\sum_{\sigma\in\Sigma} \sigma(\xi) \kappa(\sigma)^{-1}$ clearly lies in $\bz_p^{ur,\times}$ since its reduction modulo $p$ is the projection of the $\bar{\bof}_p[\Sigma]$-basis $\bar{\xi}$ of $\co_F/(p)\otimes_{\bof_p}\bar{\bof}_p$ into the $\bar{\kappa}$-eigenspace (up to the unit $|\Sigma|=f$), hence nonzero. So we find
\begin{equation}\xi_{\eta'}\in \bz_p^{ur,\times}.\label{xival}\end{equation}
We now analyze the second factor
$$ x_\eta:=\prod_{\tau\in\Sigma_\eta\backslash\Sigma}\ \sum_{\delta\in\Delta} \tau^{-1} \delta\tau(x)\cdot \eta(\delta)^{-1}$$
which is the product over the projections of $x$ into the $\eta^{p^i}$-eigenspaces for $i=0,\dots,f_\eta-1$ (up to the unit $|\Delta|=e$). For $0\leq j< e$ the $\eta_0^{-j}$- eigenspace of the inverse different is generated over $\co_F$ by $\left(\root e \of p_0\right)^{-j}$ and since $x$ was a $\co_F[\Delta]$-basis of the inverse different its projection lies in $\co_F^\times\cdot\left(\root e \of p_0\right)^{-j}$. So by Lemma \ref{expequal} below we have
$$ x_\eta\in\co_F^\times\cdot\prod_{i=0}^{f_\eta-1} \left(\root e \of p_0\right)^{-e\left\langle \frac{p^i(-m_\eta)}{e}\right\rangle}\subset K  $$
and hence
\begin{equation}v_p(x_\eta)=-\sum_{i=0}^{f_\eta-1}\left\langle\frac{-m_\eta p^i}{e}\right\rangle=-v_p(\tau(r_{\bar{\eta},\sharp},\psi_p)),\label{valeq}\end{equation}
using Corollary \ref{epsval} and the fact that $\bar{\eta}=\eta_0^{-m_\eta}$. One checks that $\tau(r_{\bar{\eta},\sharp},\psi_p)\in \bq_p^{ur}(\zeta_p)$ is an eigenvector for the character \[ \varrho=\eta_0^{-m_\eta\frac{p^{f_\eta}-1}{p-1}}\]
of the group $\Gal(\bq_p^{ur}(\zeta_p)\cap K^{ur}/\bq_p^{ur})$. Since $x_\eta$ is an eigenvector for $\varrho^{-1}$, equation  (\ref{valeq}) then implies
$$ \tau(r_{\bar{\eta},\sharp},\psi_p) \cdot x_\eta \in\bz_p^{ur,\times}.$$
Combining this with (\ref{xival}) and (\ref{epscomp}) we find
$$ \epsilon(r_{\bar{\chi}})\cdot [\per(b)]_\chi= \bar{\eta}(\mathrm{rec}(p))\tau(r_{\bar{\eta},\sharp},\psi_p)\bar{\eta}'(\sigma^{f_\eta})\cdot x_\eta\cdot \xi_{\eta'}\in \bz_p^{ur,\times} $$
and hence
$$ (\epsilon(r_{\bar{\chi}}))_{\chi\in\hat{G}}\cdot [\per(b)]\in \im(K_1(\bz_p^{ur}[G])) .$$
\end{proof}

\begin{lemma} We have $\eta=\eta_0^{m_\eta}$ where $\eta_0$ is the character (\ref{eta0def}) associated to the element $p_0$ of valuation $1$ and $m_\eta$ was defined in (\ref{metadef}).
\label{expequal}\end{lemma}

\begin{proof} It suffices to show that the composite map
\[\omega':\mu_{p^{f_\eta}-1}\subset F^\times\xrightarrow{\mathrm{rec}}G^{ab}_F\to \Gal(K/F)\xrightarrow{\eta_0^{m_\eta}}\mu_e\]
agrees with the ${m_\eta(p^{f_\eta}-1)/e}$-th power map. By definition \cite{neukirch}[Thm. V.3.1] of the tame local Hilbert symbol and the fact that our map $\mathrm{rec}$ is the inverse of that used in \cite{neukirch} we have
\[\omega'(\zeta)= \left(\frac{\zeta^{-1},p_0^{m_\eta}}{F}\right)\]
which by \cite{neukirch}[Thm. V.3.4] equals
\[ \left(\frac{\zeta^{-1},p_0^{m_\eta}}{F}\right)=\left((-1)^{\alpha\beta}\frac{p_0^\beta}{\zeta^{-\alpha}}\right)^{(p^{f_\eta}-1)/e}=\zeta^{m_\eta(p^{f_\eta}-1)/e}\]
where $\alpha=v_p(p_0^{m_\eta})=m_\eta$ and $\beta=v_p(\zeta^{-1})=0$.
\end{proof}

Denote by $\gamma$ a topological generator of $$\Gamma:=\Gal(\bq_p(\zeta_{p^\infty})/\bq_p)$$ and by \[\chi^{\cyclo}:\Gal(\bq_p(\zeta_{p^\infty})/\bq_p)\cong\bz_p^\times\]
the cyclotomic character. As in the proof of Prop. \ref{froeh} choose $b$ such that
$$\bz_p[G]\cdot b=\left(\root e\of p_0\right)^{-(e-1)}\co_ K.$$
Denote by $e_1=\frac{1}{|\Sigma|}\sum_{g\in\Sigma}g\in\bz_p[\Sigma]$ the idempotent for the trivial character of $\Sigma$.

\begin{prop} If $p\nmid |G|$ then one can choose $\beta\in H^1(K,\bz_p(1-r))$ such that
\[H^1(K,\bz_p(1-r))=H^1(K,\bz_p(1-r))_{tor}\oplus\bz_p[G]\cdot\beta\]
and the local Tamagawa number conjecture (\ref{cep}) is equivalent to the identity
\[ [(r-1)!]\cdot (p^{(r-1)c(\chi)})_{\chi\in\hat{G}}\cdot [\per(b)]^{-1}\cdot [\per(\exp^*(\beta))]\cdot [C_\beta]^{-1}\cdot\left[\frac{1-p^{r-1}\sigma}{1-p^{-r}\sigma^{-1}}\right]_F=1\]
in the group $K_1(\bq_p^{ur}[G])/\im K_1(\bz_p^{ur}[G])$. The projection of this identity into the group $K_1(\bq_p^{ur}[\Sigma])/\im K_1(\bz_p^{ur}[\Sigma])$ is
\[ [(r-1)!]\cdot [\per(\exp^*(\beta))]_F\cdot \left[\frac{\chi^{\cyclo}(\gamma)^r-1}{\chi^{\cyclo}(\gamma)^{r-1}-1}e_1+1-e_1\right] \cdot\left[\frac{1-p^{r-1}\sigma}{1-p^{-r}\sigma^{-1}}\right]_F=1\]
and in the components of $K_1(\bq_p^{ur}[G])/\im K_1(\bz_p^{ur}[G])$ indexed by $\chi=([\eta],\eta')$ with
$$\eta\vert_{\Gal(K/K\cap F(\zeta_p))}\neq 1$$
this identity is equivalent to
\begin{equation} ((r-1)!)^{f_\eta}\cdot p^{(r-1)f_\eta}\cdot [\per(b)]_\chi^{-1}\cdot [\per(\exp^*(\beta))]_\chi\in\bz_p^{ur,\times}.\label{eta-not-1}\end{equation}
\label{reform}\end{prop}

\begin{proof} If $p\nmid |G|$ then the module $H^1(K,\bz_p(1-r))/{tor}$ is free over $\bz_p[G]$ since this is true for any lattice in a free rank one $\bq_p[G]$-module. The first statement is then clear from (\ref{epscompute}) and Prop. \ref{froeh}.

Since
\[R\Gamma(K,\bz_p(1-r))\otimes^L_{\bz_p[G]}\bz_p[\Sigma]\cong R\Gamma(F,\bz_p(1-r))\]
the projection $[C_\beta]_F$ of $[C_\beta]$ into $K_1(\bq_p^{ur}[\Sigma])/\im K_1(\bz_p^{ur}[\Sigma])$ is the class of the complex
\[ H^1(F,\bz_p(1-r))_{tor}[-1]\oplus H^2(F,\bz_p(1-r))[-2]\]
and both modules have trivial $\Sigma$-action. Any finite  cyclic $\bz_p[\Sigma]$-module $M$ with trivial $\Sigma$-action has a projective resolution
$$ 0\to \bz_p[\Sigma]\xrightarrow{|M|e_1+1-e_1}\bz_p[\Sigma]\to M\to 0$$
and the class of $M$ in $K_0(\bz_p[\Sigma],\bq_p)$ is represented by $[|M|e_1+1-e_1]^{-1}\in K_1(\bq_p[\Sigma])$. Using Tate local duality we have
\begin{align*}[C_\beta]_F&=[H^1(F,\bz_p(1-r))_{tor}]^{-1}\cdot [H^2(F,\bz_p(1-r))] \\
&=[H^0(F,\bq_p/\bz_p(1-r))]^{-1}\cdot [H^0(F,\bq_p/\bz_p(r))] \\
&=[(\chi^{\cyclo}(\gamma)^{r-1}-1)e_1+1-e_1]\cdot [(\chi^{\cyclo}(\gamma)^r-1)e_1+1-e_1]^{-1}\\
&=\left[\frac{\chi^{\cyclo}(\gamma)^{r-1}-1}{\chi^{\cyclo}(\gamma)^r-1}e_1+1-e_1\right].
\end{align*}
By Prop. \ref{froeh} $[\per(b)]_\chi$ is a $p$-adic unit if $\eta=1$ which gives the second statement. The third statement follows from the fact that $\Gal(K/K\cap F(\zeta_p))$ acts trivially on
$R\Gamma(K,\bz_p(1-r))$ which implies that $[C_\beta]_\chi=1$ if the restriction of $\eta$ to $\Gal(K/K\cap F(\zeta_p))$ is nontrivial.
\end{proof}



\section{The Cherbonnier-Colmez reciprocity law}\label{sec:chercol}

Now that we have reformulated conjecture \ref{cep}
according to Prop. \ref{reform} we see
that we must compute the
image of $\exp^*(\beta)$.  In order to do this we will use an explicit
reciprocity law of \cite{chercol99}, which uses the theory of $(\varphi,
\Gamma_K)$-modules and the rings of periods of Fontaine.  Rather than
developing this machinery in full, we will give only the definitions
and results needed to state the reciprocity in our case; the reader is
invited to read \cite{chercol99} to see the theory and the reciprocity
law developed in full generality.

\subsection{Iwasawa theory}
\label{sec:iwasawa-theory}

In this subsection and the next we recall results of \cite{chercol99} specialized to the representation $V=\bq_p(1)$. For this discussion we temporarily suspend our assumption that $p\nmid |G|$. So let $K$ again be an arbitrary finite Galois extension of $\bq_p$, define
$$K_n = K(\zeta_{p^n}),\quad K_\infty = \bigcup_{n \in \N} K_n,$$
$$\Gamma_K:=\Gal(K_\infty/K),\quad \Lambda_K=\bz_p[[\Gal(K_\infty/\bq_p)]]$$
and
\[ H^m_{Iw}(K,\bz_p(1))=\varprojlim_n H^m(K_n,\bz_p(1))\cong\varprojlim_n H^m(K, \Ind_{G_{K_n}}^{G_K}\bz_p(1))\cong H^m(K,T)\]
where the inverse limit is taken with respect to corestriction maps, the second isomorphism is Shapiro's Lemma and
\[ T:=\varprojlim_n\Ind_{G_{K_n}}^{G_K}\bz_p(1)\cong \varprojlim_n \bz_p[\Gal(K_n/K)](1)\cong \bz_p[[\Gamma_K]](1)\]
is a free rank one $\bz_p[[\Gamma_K]]$-module with $G_K$-action given by $\psi^{-1}\chi^{\cyclo}$ where $$\psi:G_K\to\Gamma_K\subseteq \bz_p[[\Gamma_K]]^\times$$ is the tautological character (see the analogous discussion of (\ref{induced})). From this it is easy to see that for any $r\in\bz$ one has an exact sequence of $G_K$-modules
\begin{equation}
0\to T\xrightarrow{\gamma_K\cdot {\chi^{\cyclo}}(\gamma_K)^{r-1}-1} T \xrightarrow{\ \ \ \ } {\mathbb Z}_p(r)\to 0
\label{tseq}\end{equation}
where $\gamma_K\in\Gamma_K$ is a topological generator (our assumption that $p$ is odd assures that $\Gamma_K$ is procyclic for any $K$).  It is clear from the definition that
\begin{equation} H^m_{Iw}(K,\bz_p(1))\cong H^m_{Iw}(K_n,\bz_p(1))\label{stable}\end{equation}
for any $n\geq 0$. So $H^m_{Iw}(K,\bz_p(1))$ only depends on the field $K_\infty$, and it is naturally a $\Lambda_K$-module.
Since our base field $K$ was arbitrary an analogous sequence holds with $K$ replaced by $K_n$ and $T$ by the corresponding $G_{K_n}$-module $T_n$ so that $T\cong\Ind_{G_{K_n}}^{G_K}T_n$. In view of (\ref{stable}) we obtain induced maps
\begin{equation}\pr_{n,r}: H^1_{Iw}(K,\bz_p(1))\to H^1(K_n,\bz_p(r)) \end{equation}
for any $n\geq 0$ and $r\in\bz$.

\begin{lemma}
\label{lemma:coho_of_zpr} Set $\gamma_n=\gamma_{K_n}$.
If $r \neq 1$ then the map $\pr_{n,r}$ induces an isomorphism
$$ H^1_{Iw}(K,\bz_p(1))/(\gamma_n - {\chi^{\cyclo}}(\gamma_n)^{1-r}) H^1_{Iw}(K,\bz_p(1))\cong H^1(K_n, \bz_p(r)). $$
\end{lemma}

\begin{proof} The short exact sequence (\ref{tseq}) over $K_n$ induces a long exact sequence of cohomology groups
\[
\xymatrix{
0
\ar[r]
&
H^0_{Iw}(K,\bz_p(1))
\ar[rr]^{\gamma_n - {\chi^{\cyclo}}(\gamma_n)^{1-r}}
&&
H^0_{Iw}(K, \bz_p(1))
\ar[r]
&
H^0(K_n, {\mathbb Z}_p(r))
\ar[dlll]
\\
&
H^1_{Iw}(K,\bz_p(1))
\ar[rr]_{\gamma_n - {\chi^{\cyclo}}(\gamma_n)^{1-r}}
&&
H^1_{Iw}(K,\bz_p(1))
\ar[r]^{\pr_{n,r}}
&
H^1(K_n, {\mathbb Z}_p(r))
\ar[dlll]
\\
&
H^2_{Iw}(K,\bz_p(1))
\ar[rr]_{\gamma_n - {\chi^{\cyclo}}(\gamma_n)^{1-r}}
&&
H^2_{Iw}(K,\bz_p(1))
\ar[r]
&
H^2(K_n, {\mathbb Z}_p(r))
\ar[r]
&
0.
}
\]
By Tate local duality there is a canonical isomorphism of $\Gal(K_n/K)$-modules $$H^2(K_n, {\mathbb Z}_p(1)) \cong {\mathbb Z}_p$$ for each $n$, and the corestriction map is the identity map on $\bz_p$. Hence $$H^2_{Iw}(K,\bz_p(1))\cong\bz_p$$ with trivial action of $\Gamma_{K_n}$. This implies that for $r\neq 1$ multiplication by $\gamma_n - {\chi^{\cyclo}}(\gamma_n)^{1-r}=1 - {\chi^{\cyclo}}(\gamma_n)^{1-r}$ is injective on $H^2_{Iw}(K,\bz_p(1))$. Hence $\pr_{n,r}$ is surjective and we obtain the desired isomorphism.
\end{proof}

\subsection{The ring $A_K$ and the reciprocity law}
\label{sec:rings-a_k1-a_kpsi}

The theory of $(\varphi, \Gamma_K)$-modules \cite{chercol99} involves a ring
\[
A_K = \widehat{\co_{F'}[[\pi_K]][\frac{1}{\pi_K}]}= \left \{ \sum_{n \in \mathbb{Z} } a_n \pi_K^n: a_n \in \mathcal{O}_{F'},
  \lim_{n \to - \infty} a_n = 0 \right \},
\]
where $\pi_K$ is (for now) a formal variable and $F'\supseteq F$ is the maximal unramified subfield of $K_\infty$. The ring $A_K$ carries an operator $\varphi$ extending the Frobenius on $\co_{F'}$ and an action of $\Gamma_K$ commuting with $\varphi$ which are somewhat hard to describe in terms of $\pi_K$. However, on the subring
\[ A_{F'} = \widehat{\co_{F'}[[\pi]][\frac{1}{\pi}]} \subseteq A_K\]
one has
\begin{equation} \varphi(1 + \pi) = (1 + \pi)^p, \quad \gamma(1 + \pi) = (1 + \pi)^{\chi^{\cyclo}(\gamma)}\label{gamma-act}\end{equation}
for $\gamma\in\Gamma_K$.

The ring $A_K$ is a complete, discrete valuation ring with uniformizer $p$. We denote by $E_K\cong k((\bar{\pi}_K))$ its residue field and by $B_K = A_K[1/p]$ its field of fractions.  We see that $\varphi(B_K)$ is a subfield of $B_K$ (of degree $p$), and
thus we can define $$\psi = p^{-1} \varphi^{-1} \Tr_{B_K/ \varphi B_K}$$ and
$$\NN = \varphi^{-1} N_{B_K/\varphi B_K}$$ as further operators on $B_K$.  We
observe that if $f \in B_K$, then $$\psi (\varphi (f)) = f.$$ Thus $\psi$
is an additive left inverse of $\varphi$. We write $A_K^{\psi = 1} \subset A_K$
for the set of elements fixed by the operator $\psi$. The $(\varphi,\Gamma_K)$-module associated to the representation $\bz_p(1)$ is $A_K(1)$ where the Tate twist refers to the $\Gamma_K$-action being twisted by the cyclotomic character.

By \cite{chercol99}[III.2] the field $B_K$ is contained in a field $\tilde{B}$ on which $\varphi$ is bijective and $\tilde{B}$ contains a $G_K$-stable subring $\tilde{B}^{\dagger,n}$ consisting of elements $x$ for which $\varphi^{-n}(x)$ converges to an element in $B_{dR}$. So one has a $G_K$-equivariant ring homomorphism
\[ \varphi^{-n}: \tilde{B}^{\dagger,n}\to B_{dR} \]
which again is rather inexplicit in general but is given by
\[ \varphi^{-n}(\pi) = \zeta_{p^n} e^{t/p^n} -1\]
on the element $\pi$.

We can now summarize the main result \cite{chercol99}[Thm. IV.2.1] specialized to the representation $V=\bq_p(1)$ as follows.

\begin{theorem} Let $K/\bq_p$ be any finite Galois extension and $$\Lambda_K:=\bz_p[[\Gal(K_\infty/\bq_p)]]$$ its Iwasawa algebra.
\begin{itemize}
\item[a)] There is an isomorphism of $\Lambda_K$-modules
\[ \Exp_{\bz_p}^*: H^1_{Iw}(K,\bz_p(1))\cong A_K^{\psi=1}(1). \]

\item[b)] There is $n_0\in\bz$ so that for $n\geq n_0$ the following hold
\begin{itemize}
\item[b1)] $A_K^{\psi=1}\subseteq  \tilde{B}^{\dagger,n}$
\item[b2)] The $G_K$-equivariant map $\varphi^{-n}:A_K^{\psi=1}\to B_{dR}$ factors through
\[ \varphi^{-n}:A_K^{\psi=1}\to K_n[[t]]\subseteq B_{dR}.\]
\item[b3)] One has
\[ p^{-n}\varphi^{-n}(\Exp_{\bz_p}^*(u))= \sum_{r=1}^\infty \exp^*_{\bq_p(r)}(\pr_{n,1-r}(u))\cdot t^{r-1}\]
for any $u\in H^1_{Iw}(K,\bz_p(1))$.
\end{itemize}
\end{itemize}
\label{reciprocity}\end{theorem}

Theorem \ref{reciprocity} contains all the information we shall need when analyzing the case of tamely ramified $K$ in section \ref{sec:tame2} below. However, the paper \cite{chercol99} contains further information on the map $\Exp_{\bz_p}^*$ which we summarize in the next proposition. We shall only need this proposition when reproving the unramified case of the local Tamagawa number in section \ref{sec:unramified} below. First recall from \cite{chercol99}[p.257] that the ring $B_K$ carries a derivation
\begin{equation}  \nabla:B_K\to B_K,           \notag\end{equation}
uniquely specified by its value on $\pi$
\[ \nabla(\pi)= 1+\pi. \]
We set
\[\nabla\log(x)=\frac{\nabla(x)}{x}\]
and denote by
\[ \hat{M}:=\varprojlim_n M/p^n M\]
the $p$-adic completion of an abelian group $M$.

\begin{prop} There is a commutative diagram of $\Lambda_K$-modules where the maps labeled by $\cong$ are isomorphism.
\[
\xymatrix{
&
H^1_{Iw}(K, {\mathbb Z}_p(1))
\ar[drr]^{\Exp^*_{{\mathbb Z}_p}}_{\cong}
&&
\\
 A(K_\infty): = \varprojlim_{m,n} K_n^\times / (K_n^\times)^{p^m}
\ar[r]^-{\iota_K}_-\cong
\ar[ur]_-\delta^\cong
&
\widehat{E_K^\times}
&
\widehat{A_K^{\NN=1}}
\ar[l]_\cong^{\mod p}
\ar[r]^\cong_{\nabla \log}
&
A_K^{\psi = 1}(1)
\\
U: = \varprojlim_{m,n} \OO_{K_n}^\times / (\OO_{K_n}^\times)^{p^m}
\ar[r]^-{\iota_K\vert_U}_-{\cong}
\ar@{^(->}[u]
&
1 + \bar{\pi}_K k  [[ \bar{\pi}_K ]]
\ar@{^(->}[u]
}
\]
\label{prop:diagram}\end{prop}

\begin{proof} The isomorphism $\delta$ arises from Kummer theory.  The theory of the field of norms gives an isomorphism of multiplicative monoids \cite{chercol99}[Prop. I.1.1]
\[ \varprojlim_{n} \OO_{K_n}\xrightarrow{\cong} k[[\bar{\pi}_K]]\]
which induces our isomorphism $\iota_K\vert_U$ after restricting to units and passing to $p$-adic completions and our isomorphism $\iota_K$ by taking the field of fractions and passing to $p$-adic completions of its units.

By \cite{chercol99}[Cor. V.1.2] (see also \cite{daigle_thesis_14} 3.2.1 for more details)
the reduction-mod-$p$-map $\widehat{A_K^{\NN = 1}} \to \widehat{E_K^\times}$ is an isomorphism.

By \cite{chercol99}[Prop. V.3.2 iii)] the map $\nabla \log$ makes the upper triangle in our diagram commute. Since all other maps in this triangle are isomorphisms, the map $\nabla \log$ is an isomorphism as well.
\end{proof}

\subsection{Specialization to the tamely ramified case}\label{sec:tame1} We now resume our assumption that $p$ does not divide the degree of $[K:\bq_p]$ together with (most of) the notation from section \ref{sec:tame}. In addition we assume that
\[ \zeta_p\in K\]
which implies that $K_\infty/K$ is totally ramified and hence that $F=F'$ is the maximal unramified subfield of $K_\infty$. The theory of  fields of norms \cite{chercol99}[Rem. I.1.2] shows that $E_K$ is a Galois extension of $E_F$ of degree
\[ e:=[K_\infty:F_\infty]=[K:F(\zeta_p)]\]
with group
\[ \Gal(E_K/E_F)\cong \Gal(K_\infty/F_\infty)\cong\Gal(K/F(\zeta_p)).\]
Note that with this notation the ramification degree of $K/\bq_p$ is $e(p-1)$ whereas it was denoted by $e$ in section \ref{sec:tame}. The element $p_0$ of section \ref{sec:tame} we choose to be $-p$, i.e. we assume that $$K=F(\root e(p-1)\of{-p}).$$ An easy computation shows that $(\zeta_p-1)^{p-1}=-p\cdot u$ with $u\in 1+(\zeta_p-1)\bz_p[\zeta_p]$ and hence we can choose the root $\root (p-1)\of{-p}$ such that
\begin{equation} \zeta_p-1=\root (p-1)\of{-p}\cdot u'\label{kummerchoice}\end{equation}
with $u'\in 1+(\zeta_p-1)\bz_p[\zeta_p]$. By Kummer theory we then also have
$$K=F(\root e\of{\zeta_p-1})$$ and $B_K=B_F(\root e \of \pi)$. Any choice of $\pi_K=\root e\of \pi$ fixes a choice of
$$\root e\of{\zeta_p-1}=\varphi^{-1}(\pi_K)|_{t=0}$$
and of
$$\root e(p-1)\of{-p}=\root e\of{\zeta_p-1}\cdot(u')^{-1/e}.$$
We have
\[ G\cong \Sigma\ltimes\Delta\]
with $\Sigma$ cyclic of order $f$ and $\Delta$ cyclic of order $e(p-1)$ and
$$\Lambda_K\cong\bz_p[[G\times\Gamma_K]]\cong \bz_p[\Sigma\ltimes\Delta][[\gamma_1-1]]$$
where $\gamma_1=\gamma^{p-1}$ is a topological generator of $\Gamma_K$.

\begin{prop} There is an isomorphism of $\Lambda_K$-modules
\[ H^1_{Iw}(K,\bz_p(1))\cong \Lambda_K\cdot\beta_{Iw}\oplus\bz_p(1).\]
\label{prop:astruct}\end{prop}

\begin{proof} In view of the Kummer theory isomorphism
\[\delta: A(K_\infty)\cong  H^1_{Iw}(K,\bz_p(1))\]
it suffices to quote the structure theorem for the $\Lambda_K$-module $A(K_\infty)$ given in \cite{nsw}[Thm. 11.2.3] (where $k=\bq_p$ and our group $\Sigma\ltimes\Delta$ is the group $\Delta$ of loc.cit).
\label{free}\end{proof}

In view of Lemma \ref{lemma:coho_of_zpr} we immediately obtain the following
\begin{corollary}
\label{lemma:coho_of_zpr2}
There is an isomorphism of $\bz_p[G]$-modules
$$H^1(K, \bz_p(1-r)) \cong \bz_p[G]\cdot \beta \oplus H^1(K, \bz_p(1-r))_{tor}$$
where $\beta=\pr_{0,1-r}(\beta_{Iw})=\pr_{1,1-r}(\beta_{Iw})$.
\end{corollary}

\begin{proof} This is clear from Proposition \ref{prop:astruct} and Lemma \ref{lemma:coho_of_zpr} (with $r$ replaced by $1-r$) in view of the isomorphisms
\[ \bz_p[G]\xrightarrow{\sim}\Lambda_K/(\gamma_1 - {\chi^{\cyclo}}(\gamma_1)^{r})\Lambda_K\]
and
\begin{align*} \bz_p(1)/ (\gamma_1 - {\chi^{\cyclo}}(\gamma_1)^{r})\bz_p(1)= &\bz_p/ ({\chi^{\cyclo}}(\gamma_1) - {\chi^{\cyclo}}(\gamma_1)^{r})\bz_p\\
\cong & H^0(K,{\mathbb Q}_p/{\mathbb Z}_p(1-r)))\\
\cong & H^1(K, {\mathbb Z}_p(1-r))_{tor}.
\end{align*}
\end{proof}

If we choose the element $\beta$ of Cor. \ref{lemma:coho_of_zpr2} to verify the identity in Prop. \ref{reform} it remains to get an explicit hold on some $\Lambda_K$-basis $\beta_{Iw}$, or rather of its image
\begin{equation} \alpha=\Exp^*_{{\mathbb Z}_p}(\beta_{Iw}) \in A_K^{\psi=1}(1).\label{alphabeta}\end{equation}
Since $\alpha$ is a (infinite) Laurent series in $\pi_K$ it will be amenable to somewhat explicit analysis. In the unramified components of Prop. \ref{reform} ($\eta=1$) we can compute $\alpha$ in terms of  the well-known Perrin-Riou basis (see Prop. \ref{coleman} below) which is a main ingredient in all known proofs of the unramified case of the local Tamagawa number conjecture. In the other components ($\eta\neq 1$) we shall simply use Nakayama's Lemma to analyze $\alpha$ as much as we can in section \ref{sec:tame2}.

In order to compute $\exp^*_{\bq_p(r)}(\beta)$ we also need to be able to apply Theorem \ref{reciprocity} for $n=1$.

\begin{prop} Part b) of Theorem \ref{reciprocity} holds with $n_0=1$.
\label{recprop}\end{prop}

\begin{proof} It will follow from an explicit analysis of elements in $A_K^{\psi=1}$ in Corollary \ref{phiconv} below that $\varphi^{-1}(a)$ converges for $a\in A_K^{\psi=1}$  which shows b1). Since $\pi_K^e=\pi$ and $\varphi^{-n}(\pi)=\zeta_{p^n} e^{t/p^n}-1$ it is also clear that the values of $\varphi^{-n}$ on $A_K$, if convergent, lie in $F(\root e \of {\zeta_{p^n}-1})[[t]]=K_n[[t]]$. This shows b2). By \cite{chercol99}[Thm. IV.2.1] the right hand side of b3) is given by
$T_n \varphi^{-m}(\Exp_{\bz_p}^*(u))$ for $m\geq n$ large enough (see the next section for the definition of $T_n$). The statement in b3) then follows from Corollary \ref{sec:map-b_fpsi-=-5} below.
\end{proof}

\subsection{Some power series computations}
\label{sec:map-t_m-phi}
The purpose of this section is simply to record some computations justifying Theorem \ref{reciprocity} b3) for $n\geq 1$. Another aim is to write the coefficients of the right hand side of \ref{reciprocity} b3) in terms of the derivation $\nabla$ applied to the left hand side. First we have

\begin{lemma}
\label{lemma:del-is-dt}
  Suppose $\varphi^{-n} f$ and $\varphi^{-n} (\nabla f)$ both converge in $B_{dR}$.  Then
  \[
  \varphi^{-n} (\nabla f) = p^n \frac{d}{dt}(\varphi^{-n} (f)).
  \]
\end{lemma}

\begin{proof} This is \cite{chercol99}[Lemme III.2.3]. It's enough to check that $\varphi^{-n} \circ \nabla $ and $p^n
  \frac{d}{dt} \circ \varphi^{-n}$ both agree on $1 + \pi$, since they
  are both derivations.  We see that
  \begin{align*}
    \varphi^{-n} \nabla (1 + \pi) & = \varphi^{-n} (1 + \pi) = \zeta_{p^n}
    e^{t/p^n} \\
    p^n \frac{d}{dt} \varphi^{-n} (1 + \pi) &= p^n \frac{d}{dt}
    \zeta_{p^n} e^{t/p^n}   =  \zeta_{p^n} e^{t/p^n}.
  \end{align*}
\end{proof}

The next Lemma shows that $\nabla$ is compatible with other operators that we have introduced. The ring $B$ is defined as in \cite{chercol99}.

\begin{lemma}
\label{lemma:commutativity}

  Let $f \in B_K$.
  Then we have
  \begin{enumerate}[(a)]
  \item $  \nabla \gamma f= {\chi^{\cyclo}}(\gamma) \cdot \gamma \nabla f$ .
  \item $ \nabla \varphi f = p \cdot  \varphi \nabla f$.
  \item $\nabla \Tr_{B/\varphi B} f = \Tr_{B/\varphi B} \nabla f$.
  \item $\nabla \psi f = p^{-1} \cdot \psi \nabla f $.
  \end{enumerate}
\end{lemma}

\begin{proof}
This is a straightforward computation. For example, to see (c) note that $(1+\pi)^i$, $i=0,\dots,p-1$ is a $\varphi B$-basis of $B$ and
\[\Tr_{B/\varphi B}(x)=\Tr_{B/\varphi B}\left(\sum_{i=0}^{p-1}\varphi x_i\cdot (1+\pi)^i\right)=p\cdot \varphi x_0. \]
Hence
\begin{align*}\Tr_{B/\varphi B}(\nabla x)=&\Tr_{B/\varphi B}\left(\sum_{i=0}^{p-1}\nabla \varphi x_i\cdot (1+\pi)^i+\varphi x_i\cdot i\cdot(1+\pi)^i\right)\\
=& \Tr_{B/\varphi B}\left(\sum_{i=0}^{p-1}\varphi\left(p \nabla x_i + x_i\cdot i\right)\cdot(1+\pi)^i\right)\\
=& p^2 \varphi\nabla x_0 =\nabla(p\cdot\varphi x_0)=\nabla \Tr_{B/\varphi B}(x). \end{align*}
See \cite{daigle_thesis_14} Lemma 3.1.3 for more details.
\end{proof}

Recall the normalized trace maps
$$T_n: K_\infty \to K_n $$
from \cite{chercol99}[p.259] which are given by
$$T_n(x) = p^{-m} \Tr_{K_m/K_n} x$$
for any $m\geq n$ such that $x\in K_m$, and extend to a map
$$T_n : K_\infty[[t]] \to K_n[[t]]$$
by linearity. By \cite{chercol99}[Thm. IV.2.1] the right hand side of Theorem \ref{reciprocity} b3) is given by
$T_n \varphi^{-m}(f)$ for $f=\Exp_{\bz_p}^*(u)\in A_K^{\psi = 1}$ and $m\geq n$ large enough.  In order to get access to individual Taylor coefficients of the right hand side we wish to compute
$\frac{d^{r-1}}{dt^{r-1}} T_n \varphi^{-m} (f)$, but from Lemmas
\ref{lemma:del-is-dt} and \ref{lemma:commutativity} we see that
\[
\frac{d^{r-1}}{dt^{r-1}} T_n \varphi^{-m} = p^{-m(r-1)}T_n \varphi^{-m}\nabla^{r-1}
\]
and thus we can study the map $T_n \varphi^{-m}$ on $\nabla^{r-1}
A_K^{\psi = 1}$.  But since $\psi \nabla x = p \nabla \psi x$, we see
that $\nabla^{r-1} A_K^{\psi = 1} \subseteq A_K^{\psi = p^{r-1}} $, and so we
wish to study $T_n \varphi^{-m}$ on $A_K^{\psi = p^{r-1}}$.

\begin{lemma}
\label{lemma:3}
  Let $P \in A_K^{\psi = p^{r-1}}$ be such that $$(\varphi^{-n} P)(0):=\varphi^{-n}P\vert_{t=0}$$ converges and assume $m\geq n$.  Then if $n \geq 1$ we have
  \begin{equation}
    \label{eq:2}
  (T_n \varphi^{-m} P)(0) = p^{(r-1)m-rn}( \varphi^{-n} P)(0).
  \end{equation}
  and if $n=0$ we have
  \begin{equation}
    \label{eq:3}
    (T_0 \varphi^{-m} P)(0) = p^{(r-1)m}(1 - p^{-r} \sigma^{-1})
    (\varphi^{-0}P) (0).
  \end{equation}
\end{lemma}

\begin{proof}
  Since $P \in A_K^{\psi = p^{r-1}}$, we know that $\psi (P) = p^{r-1}
  P$ and thus
  that $$p^{-r} \Tr_{B/\varphi B}  (P) = \varphi(P).$$
  Recall that we can choose $\pi_K$ so that $\pi_K^e=\pi$. Then $\{ ((1 + \pi)
  \zeta -1)^{1/e} : \zeta \in \mu_p\}$ is the set of conjugates of
  $\pi_K$ over
  $\varphi(B)$ in an algebraic closure of $B$, so   this gives us
  \[
  p^{-r} \sum_{\zeta \in \mu_p} P ( ((1 + \pi) \zeta -1)^{1/e}) =
  P^{\sigma}( ((1 +
  \pi)^p -1)^{1/e}).
  \]
  Whenever $\varphi^{-(l+1)}P$ converges for some $l \in \N$, the operator $\varphi^{-(l+1)} P |_{t = 0}$ corresponds to setting $\pi
  = \zeta_{p^{l+1}}-1$ and applying $\sigma^{-(l+1)}$ to each coefficient. We get
  \begin{equation}
    \label{eq:1}
    p^{-r} \sum_{\zeta \in \mu_{p}} P^{\sigma^{-(l+1)}}((\zeta \cdot
    \zeta_{p^{l+1}} - 1)^{1/e}) = P^{\sigma^{-l}}((\zeta_{p^l} -1) ^{1/e}).
  \end{equation}
If $l \geq 1$, this simplifies to
  \[
  p^{-r}\Tr_{K_{l+1}/K_l}  P^{\sigma^{-(l+1)}} ((\zeta_{p^{l+1}} - 1 )^{1/e} ) =  P^{\sigma^{-l}}
  ((\zeta_{p^l} - 1)^{1/e}),
  \]
  and by induction, we see that for any $1 \leq n < m$,
  \begin{equation}
    \label{eq:4}
    p^{m-r(m-n)}T_n P^{\sigma^{-m}} ((\zeta_{p^{m}} -1)^{1/e}) = P^{\sigma^{-n}}
    ((\zeta_{p^n} -1)^{1/e}) .
  \end{equation}
  Since $P^{\sigma^{-m}} ((\zeta_{p^m} -1)^{1/e}) = (\varphi^{-m}
  P)(0)$, this proves equation (\ref{eq:2}).
If $l = 0$ then equation (\ref{eq:1}) becomes
  \[
  p^{-r} \sum_{\zeta \in \mu_{p}} P^{\sigma^{-1}}((\zeta \cdot
  \zeta_{p} - 1)^{1/e}) = (\varphi^{-0}P)(0).
  \]
  The left hand side is now equal to
  \[
    p^{-r} P^{\sigma^{-1}} (0) + p^{-r}\Tr_{K_1/K_0} P^{\sigma^{-1}} ((\zeta_p -1)^{1/e})
  \]
  and we have
  \[
  p^{-r}\Tr_{K_1/K_0}(P^{\sigma^{-1}} ((\zeta_p -1)^{1/e})) = (1 - p^{-r} \sigma^{-1}) (\varphi^{-0}P)(0).
  \]
By induction we get
   \begin{align*}
    p^{m-rm}T_0 P^{\sigma^{-m}} ((\zeta_{p^{m}} -1)^{1/e}) &= (1 - p^{-r}\sigma^{-1}) (\varphi^{-0}P)(0)
  \end{align*}
  which proves equation (\ref{eq:3}).
\end{proof}

\begin{corollary}
  \label{sec:map-b_fpsi-=-5} If $P \in A_K^{\psi = 1}$ is such that $\varphi^{-n} P$ converges and $m\geq n$ then we have
    \[
    T_n\varphi^{-m} P = p^{-n}\varphi^{-n} P
    \]
if $n\geq 1$, and
    \[
    T_0\varphi^{-m} P = (1 - p^{-1} \sigma^{-1}) \varphi^{-0}P
    \]
if $n=0$.
\end{corollary}

\begin{proof} This follows by combining Lemma \ref{lemma:3} for all $r$.
\end{proof}

\section{The unramified case}\label{sec:unramified}

In this section we reprove the local Tamagawa number conjecture (\ref{cep}) in the case where $K=F$ is unramified over $\bq_p$. This was first proven in \cite{bk88} and other proofs can be found in \cite{pr94} and \cite{bb05}. The proofs differ in the kind of "reciprocity law" which they employ but all proofs, including ours, use the "Perrin-Riou basis", i.e. the $\Lambda_F$-basis in Prop. \ref{coleman} below.

\subsection{An extension of Prop. \ref{prop:diagram} in the unramified case}
\label{sec:basis-a_fpsi-=}

In this section we use results of Perrin-Riou in \cite{pr90} to extend the diagram in Prop. \ref{prop:diagram} to the diagram in Corollary \ref{cor:2} below.  Define
\begin{align*}
{\mathcal P}_{F} := &\left \{ \sum_{n \geq 0} a_n \pi^n \in F [[ \pi
  ]]: n a_n \in {\mathcal O}_F \right \}\\
\overline {\mathcal P}_{F} := &{\mathcal P}_{F} / p {\mathcal O}_F [[ \pi ]]\\
\overline {\mathcal P}_{F, \log}  := &\{f \in \overline {\mathcal P}_{F}: (p - \varphi)(f) =
0 \} \\
{\mathcal P}_{F, \log} := &\{f \in {\mathcal P}_F: \bar{f}\in\overline {\mathcal P}_{F, \log}\}\\=&\{f \in {\mathcal P}_F: (p - \varphi)(f) \in p {\mathcal O}_F
[[ \pi ]] \}\\
{\mathcal O}_{F} [[ \pi ]]_{\log} := &\{ f \in {\mathcal O}_F [[
\pi ]]^\times  :
f \mod p {\mathcal O}_F [[ \pi ]] \in 1 + \pi k [[ \pi
]] \}\\
= &1+(\pi,p)
\end{align*}
Note that ${\mathcal P}_{F}$ is the space of power series in $F$ whose derivative with respect to
$\pi$ lies in ${\mathcal O}_F [[ \pi ]]$.  Observe that the map
$d \log$ is given by an integral power series,  and therefore
$\log {\mathcal O}_F [[ \pi ]]_{\log} \subseteq {\mathcal P}_{F}$
where the logarithm map
$$\log(1+x) = \sum_{n \geq 1} (-1)^{n-1}\frac{x^n}{n}$$
is given by the usual power series. Since $\varphi$ reduces modulo $p$ to the Frobenius, i.e. to the $p$-th power map, the logarithm series in fact induces a map
\[ \log: {\mathcal O}_F [[ \pi ]]_{\log} \to {\mathcal P}_{F,\log}.\]

We wish to show that this map is an isomorphism, and to do this we first recall Lemmas 2.1
and 2.2 from \cite{pr90}.
\begin{lemma}
\label{sec:rings-pp_f-a_f-2}
Let $$f \in 1 + \pi k [[ \pi ]]=\widehat {\mathbb G}_m (k [[ \pi
]] )$$ and let $\hat
  f$ be any lift of $f$ to ${\mathcal O}_F [[ \pi
  ]]_{\log}$.  Then $$\log (\hat f) \mod  p {\mathcal O}_F
  [[ \pi ]]\in \overline {\mathcal P}_{F,
    \log}$$ does not
  depend on the choice of $\hat f$, and the map $f \mapsto \log(\hat
  f) \mod p {\mathcal O}_F [[ \pi ]]$ is an isomorphism $\log_k:1 +
  \pi k [[ \pi ]] \rightiso \overline{{\mathcal P}}_{F, \log}$.
\end{lemma}

\begin{lemma}
\label{sec:rings-pp_f-a_f-1}
  Let $f \in {\mathcal P}_{F,\log}$.  Then the sequence $p^m \psi^m(f)$
    converges to a limit $f^\infty \in {\mathcal P}_{F, \log}$, and we have:
    \begin{enumerate}
    \item    $f^\infty \equiv f \mod p {\mathcal O}_F [[ \pi
      ]]$
    \item $\psi(f^\infty) = p^{-1} f^\infty$
    \item $(1 - p^{-1} \varphi) f^\infty \in {\mathcal O}_F [[ \pi
      ]]$
    \item $f^\infty = 0$ if $f \in {\mathcal O}_F [[ \pi ]]$
    \item $f^\infty = g^\infty$ if $f \equiv g \mod p {\mathcal O}_F [[
      \pi ]]$.
    \end{enumerate}
\end{lemma}

\begin{corollary}
\label{cor:rings-pp_f-a_f-3}
\begin{enumerate} \item The map $\log: {\mathcal O}_F [[ \pi ]]_{\log} \to {\mathcal P}_{F, \log}$ is an isomorphism.
\item One has a commutative diagram of isomorphisms
\[\xymatrix{{\mathcal O}_F [[ \pi ]]_{\log}^{\mathcal N=1}\ar[r]^{\log}_{\cong}\ar[d]^{mod p}_{\cong} &  {\mathcal P}_{F, \log}^{\psi=p^{-1}}\ar[d]^{mod p}_{\cong}\\
1 + \pi k [[ \pi ]]\ar[r]^{\log_k}_{\cong} & \overline{{\mathcal P}}_{F, \log}}
\]
\end{enumerate}
\end{corollary}

\begin{proof} To see the first part, note that we have a commutative diagram
    \[
    \xymatrix{
      1
      \ar[d]
      &
      0
      \ar[d]
      \\
      1 + p {\mathcal O}_F [[ \pi ]]
      \ar[d]
      \ar[r]^\log_\cong
      &
      p {\mathcal O}_F [[ \pi ]]
      \ar[d]
      \\
      {\mathcal O}_F [[ \pi ]]_{\log}
      \ar[d]
      \ar[r]^-\log
      &
      {\mathcal P}_{F, \log}
      \ar[d]
      \\
      1 + \pi k [[ \pi
      ]]
      \ar[d]
      \ar[r]^-{\log_k}_-\cong
      &
      \overline {{\mathcal P}}_{F, \log}
      \ar[d]
      \\
      1
      &
      0.
    }
    \]
and that the logarithm map on $1 + p {\mathcal O}_F [[ \pi ]]$ is an isomorphism since its inverse is given by the exponential series.  By the five lemma, the middle arrow is an isomorphism. To see the second part, it suffices to note that Lemma \ref{sec:rings-pp_f-a_f-1} shows that any element in  $\overline {{\mathcal P}}_{F, \log}$ has a unique lift in ${\mathcal P}_{F, \log}^{\psi=p^{-1}}$ and that $\log {\mathcal N}(x) = p \psi \log(x)$.

\end{proof}

\begin{corollary}\label{cor:2} For $K=F$ the commutative diagram from Prop. \ref{prop:diagram} extends to a commutative diagram of $\Lambda_F$-modules:
    \[
    \xymatrix{
      A(F_\infty) = \varprojlim_{m,n} F_n^\times / (F_n^\times)^{p^m}
      &
      \widehat{E_F^{\times}}
      \ar[l]_-\cong
      &
      \widehat{A_F^{{\mathcal N}=1}}
      \ar[l]_\cong^{\mod p}
      \ar[r]^\cong_{\nabla \log}
      &
      A_F^{\psi = 1}(1)
      \\
      U = \varprojlim_{m,n} {\mathcal O}_{F_n}^\times / ({\mathcal O}_{F_n}^\times)^{p^m}
      \ar@{^(->}[u]
      &
      1 + \pi k  [[ \pi ]]
      \ar[l]_-{\cong}
      \ar@{^(->}[u]
      &
      {\mathcal O}_F [[ \pi ]]_{\log}^{{\mathcal N} = 1}
      \ar[l]_\cong^{\mod p}
      \ar[r]^\cong_{\log}
      \ar@{^(->}[u]
      &
      {\mathcal P}_{F, \log}^{\psi = p^{-1}}
      \ar@{^(->}[u]^\nabla
    }
    \]
  \end{corollary}

\begin{proof} This is immediate from  Corollary \ref{cor:rings-pp_f-a_f-3} (2). \end{proof}

This diagram allows to determine the exact relationship between
${\mathcal P}_{F, \log}^{\psi = p^{-1}}$ and $A_F^{\psi = 1}(1)$ since the relationship between $A(F_\infty)$ and $U$ is quite transparent. There is an exact sequence of $\Lambda_F$-modules
\[ 0\to U\to A(F_\infty)\xrightarrow{v} \bz_p\to 0\]
where $v$ is the valuation map and $\bz_p$ carries the trivial $\Sigma\times\Gamma$-action. By
\cite{nsw}[Thm. 11.2.3], already used in the proof of Prop. \ref{prop:astruct}, there is an isomorphism
\begin{equation} A(F_\infty)\cong \Lambda_F\oplus\bz_p(1)\label{astruct}\end{equation}
and the torsion submodule $\bz_p(1)$ is clearly contained in $U$. Hence we obtain an exact sequence
\[ 0\to U_{\free}\to A(F_\infty)_{\free}\xrightarrow{v} \bz_p\to 0\]
where $M_{\free}:=M/M_{\tors}$. The module $A(F_\infty)_{\free}$ is free of rank one and since the
$\Sigma\times\Gamma$-action on $\bz_p$ is trivial we find
\[ U_{\free} = I\cdot A(F_\infty)_{\free}\]
where $$I:=(\sigma-1,\gamma-1)\subseteq \Lambda_F$$ is the augmentation ideal.

\begin{lemma} The augmentation ideal $I$ is principal, generated by  the element
  $$(1-e_1) + (\gamma-1)e_1$$
where $e_1\in\bz_p[\Sigma]$ is the idempotent for the trivial character of $\Sigma$.
\label{principal}\end{lemma}

\begin{proof} This hinges on our assumption that $p$ does not divide the order of $\Sigma$ which implies that $e_1$ has coefficients in $\bz_p$. Using $e_1^2=e_1$ we then find immediately
\begin{align*}
\sigma -1     & = (\sigma -1)(1-e_1) = (\sigma -1)(1-e_1)\cdot [ (1-e_1) + (\gamma -1) e_1 ],   \\
\gamma -1     & = \left((\gamma-1)(1-e_1) +       e_1 \right)\cdot \left[
      (1-e_1) + (\gamma -1)e_1 \right].
  \end{align*}

\end{proof}

\begin{lemma}
\label{lemma:basis}
 There are elements $\alpha \in A_F^{\psi = 1}(1)$,
$\tilde{\alpha} \in {\mathcal P}_{F, \log}^{\psi = p^{-1}}$ such that
\begin{enumerate}
\item $A_F^{\psi = 1}(1) = \Lambda_F \cdot \alpha \oplus {\mathbb Z}_p(1) \cdot 1$,
\item ${\mathcal P}_{F, \log}^{\psi = p^{-1}} = \Lambda_F \cdot \tilde{\alpha} \oplus {\mathbb Z}_p
  \cdot \log(1+ \pi)$,
\item $\nabla \tilde{\alpha} = ( (1 - e_1) + (
    \gamma  -1) e_1)\cdot\alpha$.
\end{enumerate}
\end{lemma}

\begin{proof} Part (1) follows from (\ref{astruct}) and Corollary \ref{cor:2}. For part (2) one checks easily
that ${\mathbb Z}_p \cdot \log(1 + \pi)$ is the torsion submodule of ${\mathcal P}_{F, \log}^{\psi = p^{-1}}$ and that
$({\mathcal P}_{F, \log}^{\psi = p^{-1}})_{\free}$ is free of rank one over $\Lambda_F$, since it is isomorphic under $\nabla$ to the free module
$$I\cdot\alpha=\Lambda_F\cdot((1-e_1) + (\gamma -1)e_1)\cdot \alpha$$ by Lemma \ref{principal}.  Note that we view $\alpha$ here as an element of $A_F(1)$, i.e. the action of $\gamma$ is
  ${\chi^{\cyclo}}(\gamma)$ times the standard action (\ref{gamma-act}) of $\gamma$ on $A_F$.
Setting $$\tilde{\alpha}:=\nabla^{-1}((1-e_1) + (\gamma -1)e_1)\cdot \alpha $$ we obtain (3).
\end{proof}

\subsection{The Coleman exact sequence and the Perrin-Riou basis}

Lemma \ref{lemma:basis} tells us that $\left( {\mathcal P}_{F, \log}^{\psi = p^{-1}}
\right)_{\free}$ is generated over $\Lambda_F$ by a single element
$\tilde{\alpha}$, but not what this $\tilde{\alpha}$ is.  By studying one more space,
${\mathcal O}_F [[ \pi ]]^{\psi = 0}$, we are able to describe
$\tilde{\alpha}$ and hence $\alpha$.

\begin{prop}
  \label{prop:2}
  \begin{enumerate}
  \item   There is an exact exact sequence of $\Lambda_F$-modules
    \begin{equation}
      \label{eq:10}
    0 \to {\mathbb Z}_p\cdot \log(1 + \pi) \to {\mathcal P}_{F, \log}^{\psi = p^{-1}}
    \xrightarrow{1 - \varphi/p}
    {\mathcal O}_F[[ \pi ]]^{\psi = 0} \to {\mathbb Z}_p(1) \to 0.
    \end{equation}
  \item ${\mathcal O}_F [[ \pi ]]^{\psi = 0}$ is a
    free $\Lambda_F$-module of rank 1 generated by $\xi(1 + \pi)$,
    where $\xi \in {\mathcal O}_F$ is a basis of ${\mathcal O}_F$ over ${\mathbb Z}_p
    [\Sigma]$.
  \end{enumerate}
\label{coleman}\end{prop}

\begin{proof} Part (1) is Theorem 2.3 in \cite{pr90} and goes back to Coleman's paper \cite{col79}.  See also \cite{daigle_thesis_14} Proposition 4.1.10. Part (2) is Lemma 1.5 in \cite{pr90}.
\end{proof}

\begin{corollary}
  \label{cor:3}
The bases $\alpha$ and $\tilde{\alpha}$ in Lemma \ref{lemma:basis} can be chosen such that
  \begin{equation}
    \label{eq:5}
    (1 - \varphi/p)\cdot\tilde{\alpha} = \left( (1-e_1) +
      (\gamma - {\chi^{\cyclo}}(\gamma) ) e_1 \right) \cdot \xi (1 + \pi).
  \end{equation}
\end{corollary}

\begin{proof}
  The cokernel of $(1 - \varphi/p)$ in (\ref{eq:10}) is isomorphic to $$\bz_p(1)\cong\Lambda_F
    / (\sigma-1, \gamma - {\chi^{\cyclo}}(\gamma)) $$ so the image of $(1 - \varphi/p)$ must be $\left(\sigma -1,
    \gamma - {\chi^{\cyclo}}(\gamma) \right) \cdot \xi (1 + \pi)$.  As in
  Lemma \ref{principal} we can show that this
  ideal is principal, and is generated by $$(1-e_1) + ( \gamma -
  {\chi^{\cyclo}}(\gamma))e_1.$$

\end{proof}

\subsection{Proof of the conjecture for unramified fields}
\label{sec:this-section}

We now have the tools we need to explicitly compute $\exp^*_{{\mathbb Q}_p(r)} (H^1(F,{\mathbb Z}_p(1-r)))$ and prove the equality of Proposition \ref{reform} for $K=F$ (i.e. $e=1$). By Lemma \ref{lemma:coho_of_zpr} we can take \[\beta:=\pr_{0,1-r}(\beta_{Iw})\] where $\beta_{Iw}$ satisfies
\begin{align}
\alpha=&\Exp^*_{{\mathbb Z}_p}(\beta_{Iw})\label{betadef}\\
\nabla \tilde{\alpha} = &( (1 - e_1) + (\gamma  -1) e_1)\cdot \alpha\notag\\
(1 - \varphi/p)\cdot\tilde{\alpha} = &\left( (1-e_1) + (\gamma - {\chi^{\cyclo}}(\gamma) ) e_1 \right) \cdot \xi (1 + \pi)
\notag\end{align}
using (\ref{alphabeta}), Lemma \ref{lemma:basis} (3) and (\ref{eq:5}). We cannot immediately apply Theorem \ref{reciprocity} to $n=0$ but going back to \cite{chercol99}[Thm. IV.2.1] we have
\[ \sum_{r=1}^\infty \exp^*_{{\mathbb Q}_p(r)}(\pr_{0,1-r}(u))\cdot t^{r-1}=T_0\varphi^{-m}\Exp^*_{\bz_p}(u).\]
Applying this to
\begin{equation}u=( (1 - e_1) + (\gamma  -1) e_1)\cdot\beta_{Iw}\label{uchoice}\end{equation} assures that
$$\Exp^*_{\bz_p}(u)=\nabla\tilde{\alpha}\in\co_F[[\pi]]$$
and therefore
$$\varphi^{-0}P:=\varphi^{-0}\nabla^{r-1}\Exp^*_{\bz_p}(u)=\varphi^{-0}\nabla^r\tilde{\alpha}$$
converges in $B_{dR}$ for any $r\geq 1$. Lemma \ref{lemma:3} then implies
\begin{align*} \exp^*_{{\mathbb Q}_p(r)}(\pr_{0,1-r}(u))=&\left.\frac{1}{(r-1)!}\left(\frac{d}{dt}\right)^{r-1}T_0\varphi^{-m}\Exp^*_{\bz_p}(u)\right\vert_{t=0}
\\
=&\left.\frac{1}{(r-1)!}T_0p^{-(r-1)m}\varphi^{-m}\nabla^{r-1}\Exp^*_{\bz_p}(u)\right\vert_{t=0}
\\
=&\left.\frac{1}{(r-1)!}(1-p^{-r}\sigma^{-1})\varphi^{-0}\nabla^{r}\tilde{\alpha}\right\vert_{t=0}
\\
=&\left.\frac{1}{(r-1)!}(1-p^{-r}\sigma^{-1})\nabla^{r}\tilde{\alpha}\right\vert_{\pi=0}.\end{align*}
Applying $\nabla^r$ to  (\ref{eq:5}) and using Lemma \ref{lemma:commutativity} we have
\begin{align*} (1 - p^{r-1}\varphi)\cdot\nabla^r\tilde{\alpha} = &\left( (1-e_1) + ({\chi^{\cyclo}}(\gamma)^r\gamma - {\chi^{\cyclo}}(\gamma) ) e_1 \right) \cdot \nabla^r\xi (1 + \pi)\\
=&\left( (1-e_1) + ({\chi^{\cyclo}}(\gamma)^r\gamma - {\chi^{\cyclo}}(\gamma) ) e_1 \right) \cdot \xi (1 + \pi)\end{align*}
and so we find
\begin{align*}\exp^*_{{\mathbb Q}_p(r)}(\pr_{0,1-r}(u)) =&\frac{1}{(r-1)!}\cdot\frac{1-p^{-r}\sigma^{-1}}{1 - p^{r-1}\sigma}\cdot \left( (1-e_1) + ({\chi^{\cyclo}}(\gamma)^r - {\chi^{\cyclo}}(\gamma) ) e_1 \right) \cdot \xi.\end{align*}
By Lemma \ref{lemma:coho_of_zpr} the action of $\gamma\in\Lambda_F$ on $H^1(F,\bz_p(1-r))$ is via the character $\chi^{\cyclo}(\gamma)^r$, hence for our choice (\ref{uchoice}) of $u$ we have
\begin{align*} \pr_{0,1-r}(u)= & ( (1 - e_1) + (\chi^{\cyclo}(\gamma)^r -1) e_1)\cdot \pr_{0,1-r}(\beta_{Iw})\\=&( (1 - e_1) + (\chi^{\cyclo}(\gamma)^r -1) e_1)\cdot\beta \end{align*}
and we can finally compute
\begin{align*}\exp^*_{{\mathbb Q}_p(r)}(\beta) =&\frac{1}{(r-1)!}\cdot\frac{1-p^{-r}\sigma^{-1}}{1 - p^{r-1}\sigma}\cdot \frac{ (1-e_1) + ({\chi^{\cyclo}}(\gamma)^r - {\chi^{\cyclo}}(\gamma) ) e_1}{ (1 - e_1) + (\chi^{\cyclo}(\gamma)^r -1) e_1} \cdot \xi.\end{align*}
This verifies the identity of Prop. \ref{reform}.

\section{Results in the tamely ramified case}\label{sec:tame2}  We resume our notation and assumptions from subsection \ref{sec:tame1}. Our first aim in this section is to prove Prop. \ref{explicit2} below which is a yet more explicit reformulation of the identity (\ref{eta-not-1}) in Prop. \ref{reform}. We then prove this identity for $e<p$ and $r=1$ as well as for $e<p/4$ and $r=2$. In the isotypic components where $\eta\vert_{\Gal(K/F(\zeta_p))}=1$ this can easily be done (for any $r$) using computations similar to those in subsection \ref{sec:this-section} with
\[ \beta_1:=\pr_{1,1-r}(\beta_{Iw})\]
and $\beta_{Iw}$ defined in (\ref{betadef}). The notation here is relative to the base field $K=F$. In any case, the equivariant local Tamagawa number conjecture is known for any $r$ in those isotypic components by \cite{bb05}. We shall therefore entirely focus on isotypic components with
$$\eta\vert_{\Gal(K/F(\zeta_p))}\neq 1.$$
In this case we need to verify equation (\ref{eta-not-1}). The main problem is that we do not have any closed formula for a $\Lambda_K$-basis of (the torsion free part of) $A_K^{\psi=1}$. We shall analyze a general basis using Nakayama's Lemma and to do this we first need to analyze which restrictions are put on a power series $$a=\sum_n a_n\pi_K^n\in A_K$$ by the condition $\psi(a)=a$.

\subsection{Analyzing the condition $\psi=1$} The main result of this subsection is Prop. \ref{mainestimate} below which gives the rate of convergence of $a_n\to 0$ as $n\to -\infty$ for $a\in A_K^{\psi=1}$.

\begin{definition} For $n\in\bn_0$ and $m\in\bz_{(p)}$ define
\begin{align*}  b_{m,n}:=&p^{-1}\sum_{\zeta\in\mu_p}\zeta^m(1-\zeta^{-1})^n\\=&p^{-1}\Tr_{\bq(\zeta_p)/\bq}\zeta_p^m(1-\zeta_p^{-1})^n\quad\text{  if $n\geq 1$}    \end{align*}
\end{definition}

Clearly $b_{m,n}$ only depends on $m\pmod p$.

\begin{lemma} One has $b_{m,n}\in\bz$ and
\begin{equation} b_{m,n}=\begin{cases} (-1)^{\bar{m}}{n\choose \bar{m}} & 0\leq n< p\\ (-1)^{\bar{m}}{n\choose \bar{m}} - (-1)^{\bar{m}} {n\choose \bar{m}+p} & p\leq n<2p\end{cases}
\label{bcomp}\end{equation}
where $0\leq\bar{m}<p$ is the representative for $m\pmod p$. Moreover,
$$ p^{\lfloor \frac{n+p-2}{p-1}\rfloor-1}\mid b_{m,n} $$ for $n\geq 1$ and hence $$p^j\mid b_{m,n} $$ for $j(p-1)<n\leq (j+1)(p-1)$.
\label{blemma}\end{lemma}

\begin{proof} Formula (\ref{bcomp}) follows from the binomial expansion of $(1-\zeta^{-1})^n$ and the fact that $$\sum_{\zeta\in\mu_p}\zeta^k=\begin{cases} 0 & p\nmid k\\p & p\mid k.\end{cases}$$ In particular
$ b_{m,0}=0,1$ according to whether $p\nmid m$ or $p\mid m$. The different of the extension $\bq(\zeta_p)/\bq$ is $(1-\zeta_p)^{p-2}$, so we have
\begin{align*}& \Tr_{\bq(\zeta_p)/\bq}\left(\zeta_p^m(1-\zeta_p^{-1})^n\right)\subseteq p^N\bz \\ \Leftrightarrow &\,((1-\zeta_p)^n)\subseteq\left( p^N(1-\zeta_p)^{2-p}\right)=\left((1-\zeta_p)^{N(p-1)+2-p}\right)\\  \Leftrightarrow&\, n\geq N(p-1)+2-p\Leftrightarrow N\leq\frac{n+p-2}{p-1}.
\end{align*}

\end{proof}

\begin{definition}Define integers $\beta_{n,j}\in\bz$ by $\beta_{1,j}:=\frac{1}{p}{p\choose j}
$ for $1\leq j\leq p-1$ and
\begin{equation}
\left(\sum_{j=1}^{p-1}\beta_{1,j}x^j\right)^n=\sum_{j=n}^{n(p-1)}\beta_{n,j}x^j
\notag\end{equation}
\end{definition}

\begin{prop} An element $a=\sum_i a_i\pi_K^i\in A_K$ lies in $A_K^{\psi=1}$  if and only if for all $N\in\bz$ one has
\begin{equation}
\sum_{n=0}^\infty a_{N+en}  {\frac{N}{e}+n \choose n}b_{\frac{N}{e}+n,n}= \sum_{0\leq n\leq j\leq n(p-1)}a^\sigma_{\frac{N+je}{p}}{\frac{N+je}{pe}\choose n} \beta_{n,j}\cdot p^n
\label{monster}\end{equation}
with the convention that $a_r=0$ for $r\notin\bz$. The equation (\ref{monster}) holds for all $N\in\bz$ if and only if it holds for all $N\in p\bz$.
\end{prop}

\begin{proof} This is just comparing coefficients in the identity $p^{-1}\Tr_{B/\varphi(B)}(a)=\varphi(a)$.
One has $\varphi(\pi)=(1+\pi)^p-1=\pi^p\left(1+p\cdot y\right)$ with $y=\sum_{j=1}^{p-1}\beta_{1,j}\pi^{-j}$ and hence
\[ \varphi(\pi_K)=\pi_K^p\cdot\lambda\cdot(1+p\cdot y)^{1/e}\]
with $\lambda\in\mu_e$ and $(1+Z)^{1/e}$ the binomial series. In fact, $\lambda=1$ since $\varphi(\pi_K)\equiv \pi_K^p\mod p$.
Therefore
\begin{align*} \varphi(\pi_K^m)=& \pi_K^{pm}\left(1+p\cdot y\right)^{\frac{m}{e}}=\pi_K^{pm}\sum_{n=0}^\infty {\frac{m}{e}\choose n} y^n\cdot p^n\\=
&\sum_{n=0}^\infty {\frac{m}{e}\choose n} \sum_{j=n}^{n(p-1)}\beta_{n,j}\pi_K^{pm-ej}\cdot p^n\end{align*}
and the coefficient of $\pi_K^N$ in $\varphi(a)=\sum_ma_m^\sigma\varphi(\pi_K^m)$ is
$$ \sum_{m,n,j,N=pm-ej}a^\sigma_m {\frac{m}{e}\choose n} \beta_{n,j}\cdot p^n$$
which is the right hand side of (\ref{monster}).
The conjugates of $\pi$ over $\varphi(B)$ are $(1+\pi)\zeta-1=\pi\cdot\zeta\cdot\left(1+(1-\zeta^{-1})\pi^{-1}\right)$, hence the conjugates of $\pi_K^m$ are
$$\pi_K^m\cdot\zeta^{\frac{m}{e}}\cdot\left(1+(1-\zeta^{-1})\pi^{-1}\right)^{\frac{m}{e}}=\pi_K^m\cdot\zeta^{\frac{m}{e}}\cdot\sum_{n=0}^\infty{\frac{m}{e}\choose n} (1-\zeta^{-1})^n\pi^{-n}$$
and
$$ p^{-1}\Tr_{B/\varphi(B)}(\pi_K^m)= \pi_K^m\cdot \sum_{n=0}^\infty{\frac{m}{e}\choose n} b_{\frac{m}{e},n} \pi^{-n}= \sum_{n=0}^\infty{\frac{m}{e}\choose n} b_{\frac{m}{e},n} \pi_K^{m-en} $$
and the coefficient of $\pi_K^N$ in $p^{-1}\Tr_{B/\varphi(B)}(a)$ is the left hand side of (\ref{monster}). Note here that $B(\zeta)/\varphi(B)$ is totally ramified so that all the conjugates must be congruent modulo $1-\zeta$.

Denote by (\ref{monster})$_m$ the equation (\ref{monster}) modulo $p^m$. By Lemma \ref{EK} below, (\ref{monster})$_1$ for all $N\in\bz$ is equivalent to (\ref{monster})$_1$ for all $N\in p\bz$. We shall show by induction on $m$ that this equivalence holds for all $m$. Suppose $a\in A_K$ satisfies (\ref{monster})$_{m+1}$ for all $N\in p\bz$. Let $b\in A_K^{\psi=1}$ be a lift of $\bar{a}\in E_K^{\psi=1}$ which exists by Lemma \ref{lift} below, and write $a-b=c\cdot p$. Then $a-b$ satisfies (\ref{monster})$_{m+1}$ for all $N\in p\bz$, hence $c$ satisfies (\ref{monster})$_{m}$ for all $N\in p\bz$. By induction assumption $c$ satisfies (\ref{monster})$_{m}$ for all $N\in \bz$. But then $p\cdot c$ satisfies (\ref{monster})$_{m+1}$ for all $N\in \bz$, hence so does $a=b+c\cdot p$.
\end{proof}

\begin{lemma} An element $a=\sum_i a_i\pi_K^i\in E_K$ lies in $E_K^{\psi=1}$  if and only if for all $k\in \bz$ one has
\begin{equation}
\sum_{n=0}^{p-1}a_{kp+ne}(-1)^n=a_k^\sigma.
\label{monstermodp}\end{equation}
\label{EK}\end{lemma}

\begin{proof} The only nonzero term on the right hand side of (\ref{monster})$_1$ is $a_{\frac{N}{p}}^\sigma$ corresponding to $n=j=0$, and the nonzero terms on the left hand side are for $n\leq p-1$ by Lemma \ref{blemma}. For $m\in\bz_{(p)}$ one has
$${ m\choose n}{n\choose \bar{m}} = \frac{m(m-1)\cdots(m-n+1)}{n!}\cdot\frac{n!}{\bar{m}!(n-\bar{m})!}\equiv\begin{cases} 0 & \bar{m}< n\\ 1 & \bar{m}=n\end{cases}$$
since for $\bar{m}<n$ one of the factors in $m(m-1)\cdots(m-n+1)$ is divisible by $p$ whereas for $\bar{m}=n$ this product is congruent to $\bar{m}!$ modulo $p$. For $\bar{m}>n$ one has ${ n\choose\bar{m}}=0$, so ${ m\choose n}{n\choose \bar{m}}\equiv 0$ whenever $\bar{m}\neq n$. Using (\ref{bcomp}) the left hand side of (\ref{monster})$_1$ is
$$\sum_{n=0}^{p-1} a_{N+en}  {m \choose n} {n\choose \bar{m}} (-1)^{\bar{m}}$$
for $m=\frac{N}{e}+n$. So the left hand side vanishes for $N\notin p\bz$ and is equal to the left hand side of (\ref{monstermodp}) for $N=pk$.
\end{proof}

For later reference we also record here a more explicit version of (\ref{monster})$_2$.

\begin{lemma} Let $H_0=0$ and $H_n=\sum_{i=1}^n\frac{1}{i}$ be the harmonic number. Then (\ref{monster})$_2$ holds if and only if for all $k\in\bz$ one has
\begin{equation}
\sum_{n=0}^{p-1} a_{kp+ne}(-1)^n\left(1+\frac{kp}{e}H_n\right)+\sum_{n=p+1}^{2(p-1)}  a_{kp+ne}(-1)^{n-p}\cdot p\cdot H_{n-p}\left(1+\frac{k}{e}\right) \equiv a_k^\sigma  \label{star2}\end{equation}
\label{modp2}\end{lemma}

\begin{proof} The only nonzero term on the right hand side of (\ref{monster})$_2$ for $N=kp$ is $a_{k}^\sigma$ corresponding to $n=j=0$ since for $n=1$ there is no $1\leq j\leq (p-1)$ with $p\mid (N+je)=kp+je$.  The nonzero terms on the left hand side are for $n\leq 2(p-1)$ by Lemma \ref{blemma}. Note that for $1\leq j\leq n <2p$ only $j=p$ is divisible by $p$. So computing modulo $p^2$ we have
\begin{align*} {\frac{kp}{e}+n \choose n} =& \frac{\prod_{j=1}^n \left(\frac{kp}{e}+j\right)}{n!}\equiv \frac{n!+\frac{kp}{e}\sum_{j=1}^n\frac{n!}{j}+ (\frac{kp}{e})^2\sum_{1\leq j_1<j_2\leq n}\frac{n!}{j_1j_2}}{n!}\\ \equiv &1+\frac{kp}{e} H_n + (\frac{kp}{e})^2\sum_{1\leq j_1<j_2\leq n}\frac{1}{j_1j_2} \\
\equiv & \begin{cases} 1+\frac{kp}{e} H_n & n<p \\  1+\frac{k}{e} +\frac{kp}{e}H_{n-p}+(\frac{k}{e})^2\cdot p\cdot H_{n-p} & p\leq n<2p.\end{cases}
\end{align*}
Here we have used $H_{p-1}\equiv 0\mod p$ and $\sum_{j=p+1}^n\frac{1}{j}\equiv H_{n-p}\mod p$. By (\ref{bcomp}) we have
$$ b_{\frac{kp}{e}+n,n}=\begin{cases}  {n\choose n}(-1)^n=(-1)^n & n<p\\ 0 & n=p \\ (-1)^{n-p}\left({n\choose n-p} -  {n\choose n}\right) & p< n<2p\end{cases}$$
and
$$ {n\choose n-p} -  {n\choose n} =  \frac{(p+n-p)(p+n-p-1)\cdots (p+1)}{(n-p)!}-1\equiv p\cdot \sum_{j=1}^{n-p}\frac{1}{j}. $$
So the summand for $n=p$ vanishes and for $p<n<2p$ we have
\begin{align*} {\frac{kp}{e}+n \choose n} b_{\frac{kp}{e}+n,n}\equiv &\left(1+\frac{k}{e} +\frac{kp}{e}H_{n-p}+(\frac{k}{e})^2\cdot p\cdot H_{n-p}\right) (-1)^{n-p}\cdot p\cdot H_{n-p} \\ \equiv & (-1)^{n-p}\left(1+\frac{k}{e}\right)\cdot p\cdot H_{n-p} .\end{align*}
\end{proof}

\begin{lemma} The map $A_K^{\psi=1}\to E_K^{\psi=1}$ is surjective. \label{lift}\end{lemma}

\begin{proof} This follows from the snake lemma applied to
\[\begin{CD}
0 @>>>A_K@>p>> A_K @>>> E_K @>>> 0\\
@. @V\psi-1 VV @V\psi-1 VV@V\psi-1 VV @.\\
0 @>>>A_K@>p>> A_K @>>> E_K @>>> 0
\end{CD}\]
and the fact that $A_K/(\psi-1)A_K\cong H^2_{Iw}(K,\bz_p(1))\cong \bz_p$ (see \cite{chercol99}[Rem. II.3.2.]) is $p$-torsion free.
\end{proof}

\begin{definition} For  $a=\sum_i a_i\pi_K^i\in A_K$ and $\nu\geq 1$ we set
$$ l_\nu(a):=\min\{i \ \vert\, p^\nu\nmid a_i\}.$$
In particular
$$l(a):=l_1(a)=v_{\pi_K}(\bar{a})$$ is the valuation of $\bar{a}\in E_K$.
\end{definition}
Note that $l(a)$ is independent of a choice of uniformizer for $A_K$ but $l_\nu(a)$ for $\nu\geq 2$ is not.

\begin{prop} Let $a\in A_K^{\psi=1}$.
\begin{itemize}
\item[a)] For all $\nu\geq 1$ we have $$l_\nu(a)\geq -\frac{\nu(p-1)+1}{p}\cdot e.$$ In particular $l(a)\geq -e$.
\item[b)] If $l(a)<-e+e(p-1)$ then $$l_2(a)>l(a)-e(p-1)$$
while if  $l(a)\geq -e+e(p-1)$ then $l_2(a)\geq -e$.
\item[c)] If $l(a)<-e+2e(p-1)$ and $l_2(a)\geq l(a)-e(p-1)$ then $$l_3(a)>l(a)-2e(p-1)$$
while if  $l(a)\geq -e+2e(p-1)$  and $l_2(a)\geq l(a)-e(p-1)$ then $l_3(a)\geq -e$.
\end{itemize}
\label{mainestimate}\end{prop}

\begin{remark} Part b) is a small improvement of part a) for $\nu=2$ and $a$ with $$l(a)> -(2-\frac{1}{p})e+e(p-1)$$ while part c) improves a) for $\nu=3$ and $a$ with
 $$l(a)> -(3-\frac{2}{p})e+2e(p-1)$$ and $l_2(a)\geq l(a)-e(p-1)$.
\end{remark}

\begin{proof} Suppose $a=\sum_i a_i\pi_K^i\in A_K^{\psi=1}$. Part a) is equivalent to the statement
\begin{equation} i<-\frac{\nu(p-1)+1}{p}\cdot e\, \Rightarrow\, p^\nu\mid a_i\label{snu}\end{equation}
which we denote by (\ref{snu})$_\nu$ if we want to emphasize dependence on $\nu$.  We shall prove (\ref{snu})$_\nu$ by induction on $\nu$, the statement  (\ref{snu})$_0$ being trivial. Now assume  (\ref{snu})$_{\nu'}$ for $\nu'\leq\nu$  and assume $p^{\nu+1}\nmid a_i$ for some
$$i<-\frac{(\nu+1)(p-1)+1}{p}\cdot e.$$
We shall show that there is another $i'<i$ with $p^{\nu+1}\nmid a_{i'}$. Hence there are infinitely many  $i<0$ with $p^{\nu+1}\nmid a_i$  which contradicts the fact that $a\in A_K$. This proves (\ref{snu})$_{\nu+1}$.

 In order to find $i'$ we look at the equation (\ref{monster}) for $N=pi$
\begin{equation}
\sum_{n=0}^\infty a_{pi+en}  {\frac{pi}{e}+n \choose n}b_{\frac{pi}{e}+n,n}= a_i^\sigma+\sum_{1\leq n\leq p\lambda\leq n(p-1)}a^\sigma_{i+\lambda e}{\frac{i}{e}+\lambda\choose n} \beta_{n,p\lambda}\cdot p^n
\label{star}\end{equation}
and first notice that $$p^{\nu+1-n}\mid a_{i+\lambda e}$$ for $\frac{n}{p}\leq \lambda\leq \frac{n(p-1)}{p}$. This is because of
$$ i+\lambda e< -\frac{(\nu+1)(p-1)+1}{p}\cdot e + \frac{n(p-1)}{p}\cdot e= -\frac{(\nu+1-n)(p-1)+1}{p}\cdot e $$ and the induction assumption. Since ${\frac{i}{e}+\lambda\choose n} \beta_{n,p\lambda}$ is a $p$-adic integer we conclude that $p^{\nu+1}$ divides the sum over $\lambda,n$ in the right hand side of (\ref{star}) and hence does {\em not} divide the right hand side of (\ref{star}).

Considering the left hand side of (\ref{star}) we first recall that Lemma \ref{blemma} implies that
\begin{equation}p^j\mid b_{\frac{pi}{e}+n,n}\label{div1}\end{equation}
for $j(p-1)< n\leq (j+1)(p-1)$. For $n$ in this range we have
\begin{align} pi+ne \leq\,  &  pi+(j+1)(p-1)e<-\bigl((\nu+1)(p-1)+1\bigr)e+(j+1)(p-1)e\label{ineq1}\\
=&-\bigl((\nu+1-j)(p-1)+1\bigr)e+(p-1)e\notag\\\leq &-\frac{(\nu+1-j)(p-1)+1}{p}\cdot e
\notag\end{align}
provided this last inequality holds which is equivalent to
\begin{align*}
&p\bigl((\nu+1-j)(p-1)+1\bigr)-p(p-1)\geq (\nu+1-j)(p-1)+1\\
\Leftrightarrow\, & (p-1)\bigl((\nu+1-j)(p-1)+1\bigr)\geq p(p-1)\\
\Leftrightarrow\, & \bigl((\nu+1-j)(p-1)+1\bigr)\geq p\\
\Leftrightarrow\, & (\nu+1-j)\geq 1 \Leftrightarrow \nu\geq j.
\end{align*}
So for for $1\leq j\leq\nu$ inequality (\ref{ineq1}) holds, and the induction assumption implies $$p^{\nu+1-j}\mid a_{pi+ne}.$$ Using (\ref{div1}) we conclude that $p^{\nu+1}$ divides all summands in the left hand side of (\ref{star}) except perhaps those with $n<p$ (corresponding to $j=0$). Since $p^{\nu+1}$ does not divide the right hand side, it does not divide the left hand side of (\ref{star}). So there must be one summand with $n<p$ not divisible by $p^{\nu+1}$ and hence some $i':=pi+en$ with $n\leq p-1$ so that
$p^{\nu+1}\nmid  a_{i'}$. It remains to remark that
\begin{equation}i'=pi+en \leq pi+e(p-1) < pi - i(p-1)=i \label{iprime}\end{equation}
since $i<-e$.

To prove b)  we use the same argument. Assuming the existence of $$i\leq \min\{ l(a)-e(p-1),-e-1\}$$with $p^2\nmid a_i$ we find another $i'<i$ with $p^2\nmid a_{i'}$. On the right hand side of (\ref{star}),  apart from $a_i^\sigma$, all summands are divisible by $p^2$ (note there are none with $n=1$ since $\lambda$ has to be an integer). On the left hand side, summands for $n>2(p-1)$ are divisible by $p^2$ by Lemma \ref{blemma}.  For $p\leq n\leq 2(p-1)$  we have, assuming $l(a)<-e+e(p-1)$,
\begin{align*} pi+en \leq &\, \,p\bigl(l(a)-e(p-1)\bigr)+2(p-1)e=l(a)+ (p-1)l(a)-(p-2)(p-1)e\\<\,\,&l(a)+(p-1)(-e+e(p-1))-(p-2)(p-1)e=l(a)   \end{align*}
and therefore $p\mid a_{pi+en}$. If $l(a)\geq -e+e(p-1)$ we have
\begin{align*} pi+en < & \,\,p(-e)+2(p-1)e=-e+e(p-1)\leq l(a)  \end{align*} and again conclude $p\mid a_{pi+en}$.
So all summands on the left hand side with $n\geq p$ are divisible by $p^2$. Hence some $i':=pi+en$ with $n\leq p-1$ satisfies  $p^2\nmid a_{i'}$. Moreover,  (\ref{iprime}) holds since $i<-e$.

For c) we use this argument yet another time. Assume
$$i\leq\min\{ l(a)-2e(p-1),-e-1\}$$ and $p^3\nmid a_i$.  On the right hand side of (\ref{star}) we need $p\mid a_{i+\lambda e}$ for $\frac{2}{p}\leq \lambda\leq \frac{2(p-1)}{p}$, i.e. $\lambda=1$. But $$i+e\leq \min\{l(a)-2e(p-1)+e,-1\} <l(a),$$ so $p\mid a_{i+e}$.  Assume first $l(a)<-e+2e(p-1)$. On the left hand side we have for
$p\leq n\leq 2(p-1)$
\begin{align*} pi+en \leq & \,p\bigl(l(a)-2e(p-1)\bigr)+2(p-1)e\\
=&l(a)-e(p-1)+(p-1)l(a)+e(p-1)-(2p-2)(p-1)e\\
<\,&l(a)-e(p-1)+(p-1)(-e+2e(p-1))-(2p-3)(p-1)e\\=&l(a)-e(p-1)\leq  l_2(a)\end{align*}
and therefore $p^2\mid a_{pi+en}$. For $2p-1\leq n\leq 3(p-1)$ we just add $(p-1)e$ to this last estimate to conclude
\begin{align*} pi+en \leq & \,p\bigl(l(a)-2e(p-1)\bigr)+3(p-1)e\\
<&l(a)-e(p-1)+e(p-1)=l(a) \end{align*}
and hence $p\mid a_{pi+en}$. Now assume $l(a)\geq -e+2e(p-1)$. For $p\leq n\leq 2(p-1)$ we have
\begin{align*} pi+en \leq & \,p(-e)+2(p-1)e\leq l(a)-e(p-1)\leq l_2(a) \end{align*}
and  therefore $p^2\mid a_{pi+en}$. For $2p-1\leq n\leq 3(p-1)$ we again add $(p-1)e$ to this last estimate to conclude
$pi+en<l(a)$ and $p\mid a_{pi+en}$.
As before we conclude that for some $i':=pi+en$ with $n\leq p-1$ we have $p^3\nmid a_{i'}$. Moreover  (\ref{iprime}) holds since $i<-e$.
\end{proof}

Before drawing consequences of Prop. \ref{mainestimate} we make the following definition.

\begin{definition} Let $\varpi$ be the uniformizer of $K$ given by
$$\varpi=\root e\of {\zeta_p-1}=\varphi^{-1}(\pi_K)\vert_{t=0}$$
and denote by $\vw$ the unnormalized valuation of the field $K$, i.e. $$\vw(p)=e(p-1).$$
For $a\in B_K^{\dagger,1}$ define
$$\vw(a):=\vw\left(\varphi^{-1}(a)\vert_{t=0}\right).$$
\label{vwdef}\end{definition}

\begin{corollary} For all $a\in A_K^{\psi=1}$ the series $\varphi^{-1}(a)$ converges, i.e. $A_K^{\psi=1}\subseteq B_K^{\dagger,1}$.
\label{phiconv}\end{corollary}

\begin{proof} By a) we have $p^\nu\mid a_i$ for
$$-\frac{(\nu+1)(p-1)+1}{p}\cdot e \leq i < -\frac{\nu(p-1)+1}{p}\cdot e$$
and hence
$$ v_p(a_i)\geq \nu\geq -\frac{ip+e}{e(p-1)} -1 $$
and
\begin{equation} \vw(a_i\varpi^i)\geq -(ip+e)-e(p-1)+i  =-(p-1)i-pe. \label{growth} \end{equation}
This implies
$$\lim_{i\to-\infty}\vw(a_i\varpi^i)=\infty$$
and hence the series $\sum_{i\in\bz}a_i\varpi^i$ converges in $K\subseteq \widehat{\bar{\bq}}_p$. By \cite{colmez99}[Prop. II.25] this implies that $\varphi^{-1}(a)$ converges in $B_{dR}$.
\end{proof}

\begin{prop} For each $a\in E_K^{\psi=1}$ we have $l(a)\geq -e$. If $l(a)>-e$ then $l(a)\not\equiv-e \mod p$.
Conversely, for each $c\in k^\times$ and $n\in\bz$ with $$-e<n\not\equiv-e \mod p$$ there is an element $a\in E_K^{\psi=1}$ with $l(a)=n$ and leading coefficient $c$.
\label{existence}\end{prop}

\begin{proof} That $l(a)\geq -e$ is Prop. \ref{mainestimate} a). Assume that $l(a)>-e$ and $l(a)\equiv -e\mod p$. Then $l(a)=kp+(p-1)e$ for some $k\in\bz$ and
\[ k=\frac{l(a)-(p-1)e}{p}=l(a)-\left(1-\frac{1}{p}\right)(l(a)+e)<l(a),\]
so we have $a_k=0$. Further $a_{kp+ie}=0$ for $i=0,..,p-2$ since $kp+ie<l(a)$. Hence there is only one nonzero term in (\ref{monstermodp}) which gives a contradiction.

To show the second part one can solve (\ref{monstermodp}) by an easy recursion. Alternatively, Proposition
\ref{prop:diagram} implies that $\nabla\log(a)\in E_K^{\psi=1}$ for any $a\in E_K^\times$. Now compute
\[ \nabla\log(1+c\pi_K^n)=\frac{\nabla(1+c\pi_K^n)}{1+c\pi_K^n}=\frac{cn/e\cdot(\pi_K^{n-e}+\pi_K^n)}{1+c\pi_K^n}=
\frac{cn}{e}\cdot\pi_K^{n-e}+\cdots \]
and note that for $p\nmid n$ one can produce any leading coefficient.
\end{proof}

\begin{remark} Elements $a\in E_K^{\psi=1}$ with $l(a)=-e$ exist, e.g.
\[\nabla\log(\pi^j)=j\cdot\pi^{-1}+j=j\cdot\pi_K^{-e}+j,  \]
but their leading coefficient is restricted to elements in $\bof_p$.
\label{remark:-e}\end{remark}

\begin{corollary} If $a\in A_K^{\psi=1}$ and $$l(a)<-e+e(p-1)$$ we have $\vw(a)=l(a)$.
\label{easy}\end{corollary}

\begin{proof} Since $\vw(a_{l(a)}\varpi^{l(a)})=l(a)$ we need to show $$\vw(a_i\varpi^i)>l(a)$$ for $i\neq l(a)$. This is clear for $i>l(a)$, and also for $$l(a)-e(p-1)<i<l(a)$$ since in that range $p\mid a_i$ and so $\vw(a_i\varpi^i)\geq e(p-1)+i>l(a)$. For
$$l(a)-2e(p-1)<i\leq l(a)-e(p-1)$$
we have $p^2\mid a_i$ by part b) and hence $\vw(a_i\varpi^i)\geq 2e(p-1)+i>l(a)$. Finally for
$$i\leq l(a)-2e(p-1)<-e-e(p-1)=-ep<-2e$$
we have by (\ref{growth})
$$  \vw(a_i\varpi^i)\geq -(p-1)i-pe > (p-1)2e-pe=(p-2)e>l(a)$$
using the assumption on $l(a)$.
\end{proof}

In order to study $\vw(a)$ for $a\in A_K^{\psi=1}$ with $l(a)> -e+e(p-1)$ we need to use Lemma \ref{modp2}. The next proposition will show that $\vw(a)$ cannot only depend on $l(a)$ in this case. In the situation of Prop. \ref{notsoeasy} b) one can have $\vw(a)=l(a)$ but for any $b\in A_K^{\psi=1}$ with $l(b)<l(a)-e(p-1)$ and $p^2\nmid a_{l(b)}+pb_{l(b)}$ one has $$l(a+pb)=l(a),\quad \vw(a+pb)\leq l(b)+e(p-1)<l(a)=\vw(a).$$

\begin{prop} Let $a'\in A_K^{\psi=1}$ with $$l(a')=\mnu p-e+e(p-1)$$ for some $\mu\in\bz$ with $1\leq\mnu<\frac{e(p-1)}{p}$.
\begin{itemize} \item[a)] There exists $a\equiv a' \mod p$ with
$$ l_2(a)\geq \mnu p-e=l(a)-e(p-1).$$
\item[b)] For $a$ as in a) we have $\vw(a)\geq l(a)$ with equality if $p\nmid\mnu-e$. This last condition is automatic for $e<p$.
\end{itemize}
\label{notsoeasy}\end{prop}

\begin{proof} First note that $l_2(a')\geq -e$ by Prop. \ref{mainestimate} b). If $l_2(a')=-e$ then equation (\ref{star2}) for $k:=-e$ reads
\begin{equation}
{a'}_{-e}^{\sigma}\equiv a'_{kp+e(p-1)}=a'_{-e} \notag\end{equation}
since $i=kp+en<l_2(a')$ for $n<p-1$ and $i=kp+en\leq-e+e(p-1)<l(a')$ for $p+1\leq n\leq 2(p-1)$. Hence $a'_{-e}/p\mod p\in\bof_p$. Adding an element $pb$ to $a'$, where $b$ with $l(b)=-e$ is as in Remark \ref{remark:-e}, we can assume that $l_2(a')>-e$. More generally, as long as $l_2(a')<l(a')$, we can add elements $pb$ to $a'$ whose existence is guaranteed by Prop. \ref{existence} and increase $l_2(a')$ until $l_2(a')$ is not one of the possible $l(b)$, i.e. $$l_2(a')=\mnu' p-e=(\mnu'-e)p+(p-1)e$$ for some $\mnu'\geq 1$.   Equation (\ref{star2}) for $k:=\mnu'-e$ then reads
\begin{equation}
0\equiv a'_{kp+e(p-1)}+\sum_{n=p+1}^{2(p-1)}  a'_{kp+ne}\cdot(-1)^{n-p}\cdot p\cdot H_{n-p}\left(1+\frac{k}{e}\right)  \label{star22}\end{equation}
since $i=kp+en<l_2(a')$ for $n<p-1$ and also $i=k<l_2(a')$ so that $a'_i\equiv 0$ for those $i$. If $\mnu'<\mnu$ we have for $p+1\leq n\leq 2(p-1)$  $$kp+ne < (\mnu-e)p+2(p-1)e=l(a')$$ and hence $p\mid a'_{kp+ne}$. So if $\mnu'<\mnu$ then
 $a'_{kp+e(p-1)}$ is the only non-zero term in (\ref{star22}) and we arrive at a contradiction. Therefore $\mnu'\geq\mnu$ and we have found our $a$, or otherwise we arrive at an $a$ with $l_2(a)=l(a)$. In either case this proves part a).

Equation (\ref{star22}) for $k:=\mnu-e$ gives
\begin{align}
0 \equiv & a_{kp+e(p-1)}+ a_{l(a)}\cdot (-1)\cdot p\cdot H_{p-2}\left(1+\frac{\mnu-e}{e}\right)\notag \\
\equiv & a_{kp+e(p-1)} - a_{l(a)}\cdot p\cdot \frac{\mnu}{e} \pmod {p^2}\label{star222}\end{align}
since $p\mid a_{kp+ne}$ for $kp+ne<kp+2(p-1)e=l(a)$. Note also
$$ H_{p-2}=H_{p-1}-\frac{1}{p-1}\equiv 0-(-1)=1 \pmod p.$$
For part b) we need to show that $\vw(a_i\varpi^i)\geq l(a)$ for all $i\in\bz$ (and compute the sum over those $i$ for which there is equality). As in the proof of Corollary \ref{easy} for $i>l(a)$ and $l(a)-e(p-1)<i<l(a)$ we obviously have $\vw(a_i\varpi^i)> l(a)$. By (\ref{star222}) we have
\begin{align}a_{l(a)-e(p-1)}\varpi^{l(a)-e(p-1)}+a_{l(a)}\varpi^{l(a)}\equiv &\left(\frac{p\mnu}{\varpi^{e(p-1)}e}+1\right)a_{l(a)}\varpi^{l(a)}\notag
\\=&\left(-\frac{\mnu}{e}+1\right)a_{l(a)}\varpi^{l(a)}+O(\varpi^{l(a)+1})\label{twoterms}\end{align}
since
$$\varpi^{e(p-1)}=(\zeta_p-1)^{p-1}\equiv -p \pmod {(\zeta_p-1)^p}.$$
So if $p\nmid -\frac{\mnu}{e}+1$ this is the leading term of valuation $l(a)$. For $$l(a)-2e(p-1)<i < l(a)-e(p-1),$$ since $l_2(a)\geq l(a)-e(p-1)$ by part a), we have $p^2\mid a_i$ and hence $\vw(a_i\varpi^i)\geq 2e(p-1)+i>l(a)$.
For
$$l(a)-3e(p-1)< i\leq l(a)-2e(p-1)$$
we have $p^3\mid a_i$ by c) of Prop. \ref{mainestimate} and hence $\vw(a_i\varpi^i)\geq 3e(p-1)+i>l(a)$ .
Finally for $$i\leq l(a)-3e(p-1)<-e-e(p-1)=-ep$$
we have by (\ref{growth})
$$  \vw(a_i\varpi^i)\geq -(p-1)i-pe > (p-1)pe-pe=(p-2)pe\geq (2p-3)e>l(a)$$
using the assumption on $l(a)$.
\end{proof}


\subsection{Isotypic components}\label{sec:isotypic} We introduce some notation for isotypic components. Recall that
\[ G\cong \Sigma\ltimes\Delta\]
with $\Sigma$ cyclic of order $f$ and $\Delta$ cyclic of order $e(p-1)$. For any $\Sigma$-orbit $[\eta]$ we define the idempotent
$$ e_{[\eta]}=\sum_{\eta'\in\widehat{\Sigma_\eta}}e_\chi\in\bz_p[G]$$
where the irreducible characters $\chi=([\eta],\eta')$ of $G$ are parametrized as in section \ref{sec:tame}. For any $\bz_p[G]$-module $M$ its $[\eta]$-isotypic component
\[ M^{[\eta]}:=e_{[\eta]}M\]
is a again a $\bz_p[G]$-module. The $\Sigma$-orbit \begin{equation}[\eta]=\{\eta,\eta^p,\eta^{p^2},\dots,\eta^{p^{f_\eta-1}}\}=\{\eta_0^{n_1},\dots,\eta_0^{n_{f_\eta}}\}\label{orbit}\end{equation} corresponds to an orbit $\{n_1,\dots,n_{f_\eta}\}\subseteq\bz/e(p-1)\bz$ of residue classes modulo $e(p-1)$ under the multiplication-by-$p$ map, i.e. we have $n_{i+1}\cong n_ip\mod e(p-1)$ where we view the index $i$ as a class in $\bz/f_\eta\bz$. We shall use the notation
\[[\eta]=\{n_1,\dots,n_{f_\eta}\}=[n_i]\]
to denote both the orbit of residue classes in $\bz/e(p-1)\bz$ and the orbit of characters. By (\ref{kummerchoice})  the group
\[ \Delta_e:=\Gal(K/F(\zeta_p))\]
acts on $\root e\of {\zeta_p-1}=\varphi^{-1}(\pi_K)|_{t=0}$ via the character $\eta_0$ defined in section \ref{sec:tame} and acts on $\pi_K$ via $\eta_0^p$. The $[\eta]=\{n_1,\dots,n_{f_\eta}\}$-isotypic component of the $\bz_p[\Sigma\ltimes \Delta_e]$-module $A_K$ is
\begin{equation} \{a=\sum a_n\pi_K^n\vert \text{ $a_n=0$ for $n\mod e\notin\{n_1,\dots,n_{f_\eta}\}$}\}\notag\end{equation}
but $A_K^{[\eta]}$ is much harder to describe since $\pi_K$ is not an eigenvector for the full group $\Delta$. However, there is  the following fact about leading terms.

\begin{lemma} Fix $\nu\geq 1$, $a=\sum_j a_j\pi_K^j\in A_K$ and denote by $e_\eta\in\co_F[\Delta]$ the idempotent for $\eta=\eta_0^n$. If
\begin{equation} p\cdot l_\nu(a) \equiv n\mod e(p-1)\label{lcong}\end{equation} then
\[ l_\nu(e_\eta a)=l_\nu(a)\]
and the leading coefficients modulo $p^\nu$ of $e_\eta a$ and $a$ agree.
If $a=e_\eta a$ is an eigenvector for $\Delta$ then (\ref{lcong}) holds.
\label{eigen}\end{lemma}

\begin{proof} Denote by
\[ \omega:\Delta\to\Gal(F(\zeta_p)/F)\to\bz_p^\times\]
the Teichmueller character. For $\delta\in\Delta$ we have
\begin{align*} \delta(\pi_K)=&\left((1+\pi)^{\omega(\delta)}-1\right)^{1/e}=\left(\sum_{i=1}^\infty {{\omega(\delta)}\choose {i}}\pi^i\right)^{1/e}\\
=&\lambda(\delta)\pi_K\left(1+\sum_{i=2}^\infty \frac{1}{\omega(\delta)}{{\omega(\delta)}\choose {i}}\pi^{i-1}\right)^{1/e}
\end{align*}
where $\lambda(\delta)\in\mu_{e(p-1)}$ satisfies $\lambda(\delta)^e=\omega(\delta)$ and $(1+Z)^{1/e}$ denotes the usual binomial series. Applying $\varphi^{-1}\vert_{t=0}$ we find
\begin{align*} \delta(\root e\of {\zeta_p-1})\equiv&\lambda(\delta)^{1/p}\cdot\root e\of {\zeta_p-1}\mod \varpi^2 \end{align*}
and since $\root e\of {\zeta_p-1}\equiv\root {e(p-1)} \of {-p} \mod \varpi^2$ we obtain $\lambda(\delta)=\eta_0(\delta)^p$.
In particular, for any $a\in A_K$
\[ \delta(a)\equiv \eta_0(\delta)^{p\cdot l_\nu(a)}\cdot a_{l_\nu(a)}\cdot\pi_K^{l_\nu(a)}\mod (p^\nu,\,\pi_K^{l_\nu(a)+1})\]
and
\begin{align*} e_\eta a=&\frac{1}{e(p-1)}\sum_{\delta\in\Delta}\eta^{-1}(\delta)\delta(a)\equiv\frac{1}{e(p-1)}\sum_{\delta\in\Delta} \eta_0(\delta)^{p\cdot l_\nu(a)-n}\cdot a_{l_\nu(a)}\cdot\pi_K^{l_\nu(a)}\\
&\equiv\begin{cases}a_{l_\nu(a)}\cdot\pi_K^{l_\nu(a)} & \text{if $p\cdot l_\nu(a) \equiv n\mod e(p-1)$}\\
0& \text{if $p\cdot l_\nu(a) \not\equiv n\mod e(p-1)$} \end{cases}\end{align*}
where the congruences are modulo $(p^\nu,\,\pi_K^{l_\nu(a)+1})$. This implies both statements in the lemma.
\end{proof}

\begin{remark} With the notation introduced in this section we have
\[ e_{[\eta]}=\sum_{i=1}^{f_\eta}e_{\eta^{p^i}}.\]
\end{remark}

\subsection{The main result} We view $\Sigma$ is as a subgroup of $G$ in such a way that $\root e(p-1)\of {-p}\in K^\Sigma$ where $\root e(p-1)\of {-p}$ is the choice of root corresponding to our choice of root $\pi_K$ of $\pi$. Then the $\bz_p[\Sigma]$-algebra $\bz_p[G]$ is finite free of rank $e(p-1)$. For each choice of $\eta$ the $[\eta]$-isotypic component of $\bz_p[G]$ is free of rank $f_\eta$ over $\bz_p[\Sigma]$ and for each $\eta\neq \omega$ the $[\eta]$-isotypic component $$(A_K^{\psi=1}(1))^{[\eta]}$$ of $A_K^{\psi=1}(1)$ is free of rank $f_\eta$ over $\bz_p[\Sigma][[\gamma_1-1]]$. Write
$$[\eta]=\{n_1,\dots,n_{f_\eta}\}=[n_1]\subseteq\bz/e(p-1)\bz$$ and pick representatives $n_i\in\bz$ with \begin{equation} 0 < n_i < e(p-1),\quad i=1,\dots,f_\eta.\notag\end{equation}
Note that our running assumption $\eta\vert_{\Delta_e}\neq 1$ implies $e\nmid n_i$.

\begin{prop} Fix $\eta\vert_{\Delta_e}\neq 1$ and let $\{\alpha_i\vert\ i=1,\dots,f_\eta\}$ be a $\bz_p[\Sigma][[\gamma_1-1]]$-basis of $(A_K^{\psi=1}(1))^{[\eta]}$. Let $n_{i,r}$
be representatives for the residue classes $[n_1-re]\subseteq\bz/e(p-1)\bz$ with
$$0<n_{i,r}<e(p-1)$$
indexed such that $n_i-re\equiv n_{i,r}\mod e(p-1)$. Consider the two $\bz_p[\Sigma]$-lattices
\[ L_r:=\bigoplus_{i=1}^{f_\eta}\bz_p[\Sigma]\cdot(\nabla^{r-1}\alpha_i^{\sigma^{-1}})(\root e \of {\zeta_p-1})\]
and
\[ \co_{K}^{[n_1-re]}=\bigoplus_{i=1}^{f_\eta}\co_F\cdot(\root {e(p-1)} \of {-p})^{n_{i,r}}\]
in the $[n_1-re]$-isotypic component
\[ K^{[n_1-re]}=\bigoplus_{i=1}^{f_\eta}F\cdot(\root {e(p-1)} \of {-p})^{n_{i,r}}=\bigoplus_{i=1}^{f_\eta}F\cdot(\root {e(p-1)} \of {-p})^{n_i-re}\]
of $K$.
Then the conjunction of (\ref{eta-not-1}) (in Prop. \ref{reform}) for $\chi=([n_1-re],\eta')$ over all $\eta'$ holds if and only if
$L_r$ and $\co_{K}^{[n_1-re]}$ have the same $\bz_p[\Sigma]$-volume, i.e.
\begin{equation}\Det_{\bz_p[\Sigma]}L_r=\Det_{\bz_p[\Sigma]}\co_{K}^{[n_1-re]}\label{sigmavol}\end{equation}
inside $\Det_{\bq_p[\Sigma]}{K}^{[n_1-re]}$.
\label{explicit2}\end{prop}

\begin{proof} Let $\alpha$ be a $\Lambda_Ke_{[n_1]}$-basis of $(A_K^{\psi=1}(1))^{[n_1]}$. Then
\[ \beta_{Iw}:=(\Exp^*_{\bz_p})^{-1}(\alpha)\]
is a $\Lambda_Ke_{[n_1]}$-basis of $H^1_{Iw}(K,\bz_p(1))^{[n_1]}$ and the element
$$\beta=\pr_{1,1-r}(\beta_{Iw})$$ of Corollary \ref{lemma:coho_of_zpr2} is a
$\bz_p[G]e_{[n_1-re]}$-basis of $(H^1(K,\bz_p(1-r))/{tor})^{[n_1-re]}$. This follows from the fact that the isomorphism
$\pr_{1,1-r}$ of Lemma \ref{lemma:coho_of_zpr} is not $\Lambda_K$-linear but $\Lambda_K$-$\kappa_{-r}$-semilinear where $\kappa_j$ is the automorphism of $\Lambda_K$ given by $g\mapsto g\chi^{\cyclo}(g)^{j}$ for $g\in G\times\Gamma_K$.
Theorem \ref{reciprocity} and Prop. \ref{recprop} imply
\begin{align*} \exp_{\bq_p(r)}^*(\beta)=&\frac{1}{(r-1)!}\left(\frac{d}{dt}\right)^{r-1}p^{-1}\varphi^{-1}(\alpha)\vert_{t=0}\\
=&\frac{p^{-r}}{(r-1)!}(\nabla^{r-1}\alpha^{\sigma^{-1}})(\root e \of {\zeta_p-1}).
\end{align*}
Hence the $\bz_p[G]e_{[n_1-re]}$-lattice
\begin{equation}\bz_p[G]\cdot(r-1)!\cdot p^{r-1}\cdot\exp_{\bq_p(r)}^*(\beta)\subset K^{[n_1-re]}\label{lattice}\end{equation}
is free over $\bz_p[\Sigma]$ with basis
\[ (r-1)!\cdot p^{r-1}\cdot\frac{p^{-r}}{(r-1)!}(\nabla^{r-1}\alpha_i^{\sigma^{-1}})(\root e \of {\zeta_p-1})=p^{-1}\cdot(\nabla^{r-1}\alpha_i^{\sigma^{-1}})(\root e \of {\zeta_p-1})\]
where $i=1,\dots,f_\eta$. Now the conjunction of (\ref{eta-not-1}) for $\chi=([n_1-re],\eta')$ over all $\eta'$ is equivalent to the statement that the lattice (\ref{lattice}) and the $[n_1-re]$-isotypic component of the inverse different
\[ (\root e \of {\zeta_p-1})^{-(e(p-1)-1)}\co_K\]
have the same $\bz_p[\Sigma]$-volume. Since $e\nmid n_1$ we have
\[\left((\root e \of {\zeta_p-1})^{-(e(p-1)-1)}\co_K\right)^{[n_1-re]}=\left(p^{-1}\co_K\right)^{[n_1-re]}\]
and the statement follows.
\end{proof}

\subsection{Proof for $r=1,2$ and small $e$} We retain the notation of the previous section. As in Prop. \ref{prop:2} denote by $\xi$ a $\bz_p[\Sigma]$-basis of $\co_F$.

\begin{prop}  There exists a $\bz_p[\Sigma][[\gamma_1-1]]$-basis
$$\alpha_i=\xi\cdot\pi_K^{l(\alpha_i)}+\cdots \in A_K^{\psi=1},\quad i=1,\dots,f_\eta$$
of $(A_K^{\psi=1})^{[n_1-e]}$ with
$$ l(\alpha_i)=\begin{cases} n_i-e & \text{if $p\nmid n_i$}\\ n_i-e+e(p-1) & \text{if $p\mid n_i$.}\end{cases}$$
\label{tamebasis}\end{prop}

\begin{proof} By Nakayama's Lemma it suffices to find a $\bof_p[\Sigma]$-basis for \begin{equation}(A_K^{\psi=1})^{[n_1-e]}/(p,\gamma_1-1)\cong \left(A_K^{\psi=1}/(p,\gamma_1-1)\right)^{[n_1-e]}.\label{isotypic2}\end{equation} By Lemma \ref{lift} we have $A_K^{\psi=1}/pA_K^{\psi=1}=E_K^{\psi=1}$. By Prop. \ref{existence} (reductions mod $p$ of) elements $\alpha_i$ as described in Prop. \ref{tamebasis} exist in $E_K^{\psi=1}$. By projection and Lemma \ref{eigen} we can also assume that they are in the $[n_1-e]$-isotypic component. Let $a'$ be a nonzero $\bz_p[\Sigma]$-linear combination of the $\alpha_i$ and assume $$a'\equiv (\gamma_1-1)a \mod p$$ for some $a\in A_K^{\psi=1}$. By Lemma \ref{gammaaction} below we have $l(a')\geq-e+e(p-1)$. Since $l(a')=l(\alpha_i)$ for some $i$, this implies $$l(a')\equiv -e+e(p-1)\equiv -2e \mod p.$$ Using Lemma \ref{gammaaction} again we have $l(a)\leq l(a')-e(p-1)\equiv -e\mod p$. Since $l(a)\not\equiv -e\mod p$ by Prop. \ref{existence} we have strict inequality. Lemma \ref{gammaaction} then shows $p\mid l(a)$ and hence $p\mid l(a')$, contradicting $l(a')\equiv -2e\mod p$. We conclude that the $\alpha_i$ are linearly independent in (\ref{isotypic2}). Since the $\bof_p[\Sigma]$-rank of
(\ref{isotypic2}) is $f_\eta$ this finishes the proof.
\end{proof}

\begin{lemma} For $a\in E_K^{\psi=1}$ with $l(a)=j p^\kappa$ with $p\nmid j$ we have
\[ l((\gamma_1-1)a)=(j+e(p-1))p^\kappa.\]
In particular
\[ l((\gamma_1-1)a)\geq l(a)+e(p-1)\]
with equality if and only if $p\nmid l(a)$, and
\[ l((\gamma_1-1)a)\geq -e+e(p-1)\]
for all $a\in E_K^{\psi=1}$.
\label{gammaaction}\end{lemma}

\begin{proof} Since $\chi^{\cyclo}(\gamma_1)=1+p$ we find from (\ref{gamma-act}) that (in $E_K$)
\[ \gamma_1(\pi)=\pi+\pi^p+\pi^{p+1}\]
and hence for $n=jp^\kappa$
\begin{align*} (\gamma_1-1)\pi_K^n=&(\pi+\pi^p+\pi^{p+1})^{n/e}-\pi^{n/e}=\pi_K^n\left((1+\pi^{p-1}+\pi^p)^{n/e}-1\right)\\
=&\pi_K^n\left((1+\pi^{p^\kappa(p-1)}+\pi^{p^{\kappa+1}})^{j/e}-1\right)\\
=&\frac{j}{e}\cdot\pi_K^{n+ep^\kappa(p-1)}+\cdots
\end{align*}
and this is indeed the leading term since $p\nmid j$. The last assertion follows from Prop. \ref{mainestimate} a).
\end{proof}

\begin{prop} If $e<p$ the identity (\ref{sigmavol}) holds for $r=1$.
\label{r1prop}\end{prop}

\begin{proof} We first remark that for each $i$ we have
$$ \vw(\alpha_i)=l(\alpha_i)=\begin{cases} n_i-e & \text{if $p\nmid n_i$}\\ n_i-e+e(p-1) & \text{if $p\mid n_i$}\end{cases}$$
by Corollary \ref{easy} and Proposition \ref{notsoeasy}. Note that there is at most one $n_i$, $n_1$ say, with
\[ 0<n_1\leq e-1 \]
since all the $n_i$ lie in the same residue class modulo $p-1$ and $e\leq p-1$. Then
$$n_2=pn_1\leq ep-p<ep-e=e(p-1)$$
and conversely, $p\mid n_2$ if and only if $0<n_1:=n_2/p\leq e-1.$
For all other $i$ we have $n_i-e=n_{i,1}$. So if no $n_i-e$ is negative then
$$q_i:=\alpha_i^{\sigma^{-1}}(\root e \of {\zeta_p-1})\in K$$
is already a basis of $\co_{K}^{[n_1-e]}$. Otherwise
\[ p \cdot q_1, p^{-1}\cdot q_2, q_3,\dots, q_{f_\eta} \]
is a basis of $\co_{K}^{[n_1-e]}$. Since $L_1$ is the span of the $q_i$ the statement follows.
\end{proof}

\begin{remark} Although not covered by Prop. \ref{prop:cep} it is in fact true that the equivariant local Tamagawa number conjecture for $r=1$ is equivalent to (\ref{sigmavol}) for $r=1$ and so Prop. \ref{r1prop} proves this conjecture for $e<p$. However, for $r=1$ one can give a direct proof without any assumption on $e$ other than $p\nmid e$ by studying the exponential map instead of the dual exponential map. Since the exponential power series gives a $G$-equivariant isomorphism $$\exp:p\cdot\co_{K}\cong 1+p\cdot\co_{K}$$ one can easily compute the (equivariant) relative volume of $\exp(\co_{K})$ and $\widehat{\co_{K}^\times}\subseteq H^1(K,\bz_p(1))$. For more work on the case $r=1$ see \cite{bleycobbe16} and references therein.
\end{remark}

To prepare for the proof of Prop. \ref{r2prop} below we need to compute $\vw(\nabla \alpha_i)$, i.e. prove the analogues of Corollary \ref{easy} and Prop. \ref{notsoeasy} for $\nabla a\in A_K^{\psi=p}$.

\begin{lemma} Assume $e<p/2$. For $a\in A_K^{\psi=1}$ with $$p\nmid l(a)<-e+e(p-1)$$ or with $$l(a)=\mu p-e+e(p-1)$$ and chosen as in Prop. \ref{notsoeasy} a) we have
\begin{equation} \vw(\nabla a)=l(\nabla a)=l(a)-e.
\notag
\end{equation}
\label{notdivbyp}\end{lemma}

\begin{proof} Since
\begin{equation} \nabla \pi_K^j=\frac{j}{e}\pi_K^{j-e}+\frac{j}{e}\pi_K^{j}               \label{nablaaction}\end{equation}
it is clear that $l(\nabla a)=l(a)-e$ if $p\nmid l(a)$. To compute $\vw(\nabla a)$ note that from the proof of Cor. \ref{easy}  we already know
$$\vw(a_j\varpi^j)>l(a)$$ for $j\neq l(a)$. But this implies
\begin{equation}\vw\left(a_j\frac{j}{e}\varpi^{j-e}\right)>l(a)-e,\quad \vw\left(a_j\frac{j}{e}\varpi^{j}\right)>l(a)>l(a)-e\label{nablaest}\end{equation} for $j\neq l(a)$.
This finishes the proof for the case $p\nmid l(a)<-e+e(p-1)$. If $$l(a)=\mu p-e+e(p-1)$$ then recall from the proof of
Prop. \ref{notsoeasy} b) that we had to compute modulo $p^2$ and there were two terms in (\ref{twoterms}) with valuation $l(a)$ arising from $j=l(a)$ and $j=l(a)-e(p-1)$. Normalizing the leading coefficient to be $\xi$ (as in the $\alpha_i$) we have
\[ a\equiv\xi\cdot\frac{\mu p}{e}\cdot\pi_K^{l(a)-e(p-1)} + \cdots + \xi\cdot \pi_K^{l(a)}+\cdots \mod p^2\]
and
\[ \nabla a\equiv\xi\cdot\frac{\mu p}{e}\cdot\frac{\mu p-e}{e}\cdot\pi_K^{l(a)-e-e(p-1)} + \cdots + \xi\cdot\frac{l(a)}{e}\cdot \pi_K^{l(a)-e}+\cdots \mod p^2\]
and hence
\begin{align*}&\frac{\mu p}{e}\cdot\frac{\mu p-e}{e}\cdot\varpi^{l(a)-e-e(p-1)} + \frac{l(a)}{e}\cdot \varpi^{l(a)-e}\\
\equiv&\left(-\frac{\mu}{e}\cdot\frac{\mu p-e}{e}+ \frac{l(a)}{e} \right)\cdot \varpi^{l(a)-e} \mod p^2.\end{align*}
Computing the leading coefficient modulo $p$ we find
\[ \left(\frac{\mu}{e}+ \frac{-2e}{e} \right)=\frac{\mu}{e}-2 \]
which is divisible by $p$ if and only if $p\mid \mu-2e$. Since $e<p/2$ we have
$$-p<-2e < \mu-2e < \frac{e(p-1)}{p}-2e=\left(-1-\frac{1}{p}\right)e<0$$
and hence $p\nmid\mu-2e$.
In the proof of Prop. \ref{notsoeasy} b) we showed $\vw(a_j\varpi^j)>l(a)$ for $j\neq l(a),l(a)-e(p-1)$ and as above this implies that the corresponding terms in $\nabla a$ all have valuation larger than $l(a)-e$.
\end{proof}

We handle the case $p\mid l(a)$ in a separate Lemma. Similar to Prop. \ref{notsoeasy} we need to compute modulo $p^2$.

\begin{lemma} Assume $e<p/4$ and $0< \mu p <-e+e(p-1)$. Then there exists $a\in(A_K^{\psi=1})^{[\mu p]}$ with $l(a)=\mu p$ and
\[ \vw(\nabla a)=l(\nabla a)=\mu p-e+e(p-1).\]
Moreover we can choose $a$ with any leading coefficient.
\label{divbyp}\end{lemma}

\begin{proof} The statement about the leading coefficient will be clear from the proof, so to alleviate notation we take the leading coefficient to be $1$. First we can find $a'\in A_K^{\psi=1}$ with
\[ a'\equiv\pi_K^{\mu p}-\pi_K^{\mu p+e(p-1)}+\cdots \pmod {p^2},\]
i.e. with $a_i'\equiv 0$ for all $i<\mu p+e(p-1)$ and $i\neq \mu p$. To see this, first note that (\ref{star2}) is satisfied for $k=\mu$ since $H_{p-1}\equiv 0 \pmod p$ (and we take $a_{\mu p+ne}'$ arbitrary but divisible by $p$ for $n=p+1,\dots, 2(p-1)$). In any equation (\ref{star2}) with index $k<\mu$ the coefficient $a_{\mu p}'$ does not occur on the left hand side since $kp+ne$ is a multiple of $p$ only for $n=0$ among $n\in\{0,\dots,p-1,p+1,\dots,2(p-1)\}$. On the right hand side we always have $a_k'\equiv 0$ since $k<\mu<\mu p$. Similarly, the coefficient $a_{\mu p+e(p-1)}'$ does not occur on the left hand side for $k<\mu$ since $kp+ne=\mu p+e(p-1)$ implies $n \equiv -1 \pmod p$, i.e. $n=p-1$. So the fact that $a_i'\not\equiv 0$ for $i=\mu p,\mu p+e(p-1)$ forces no further nonzero terms in equations with index $k<\mu$. Equations (\ref{star2}) with index $k>\mu$ can always be satisfied inductively by adjusting the variable $a'_{kp+(p-1)e}$ since $a_{kp+(p-1)e}'$ does not occur in any equation with index $k'<k$.

With the notation introduced in subsection \ref{sec:isotypic} set
\[ a=e_{[\mu p]}a' \in (A_K^{\psi=1})^{[\mu p]}\]
so that $l(a)=l_2(a)=\mu p$ by Lemma \ref{eigen}. We have
\[\nabla a'\equiv \frac{\mu p}{e}\cdot\pi_K^{\mu p-e}+\frac{\mu p}{e}\cdot\pi_K^{\mu p}-\frac{\mu p+e(p-1)}{e}\cdot\pi_K^{\mu p-e+e(p-1)}+\cdots \pmod {p^2}\]
and hence
\begin{align*}\nabla a=&\nabla e_{[\mu p]}a'=e_{[\mu p-e]}\nabla a'\\
\equiv &\frac{\mu p}{e}\cdot\pi_K^{\mu p-e}+\cdots-\left(\frac{\mu p+e(p-1)}{e}-\frac{\mu p}{e}x\right)\cdot\pi_K^{\mu p-e+e(p-1)}+\cdots\pmod {p^2}
\end{align*}
where $x$ is the coefficient of $\pi_K^{\mu p-e+e(p-1)}$ in the expansion of $e_{[\mu p-e]}(\pi_K^{\mu p-e}+\pi_K^{\mu p})$. Moreover
\[ l(\nabla a)=l(\nabla e_{[\mu p]}a')=l(e_{[\mu p-e]}\nabla a')=l(\nabla a')=\mu p-e+e(p-1).\]
In order to show that $\vw(\nabla a)=l(\nabla a)$ write
\[\nabla a=\sum_ib_i\cdot\pi_K^i.\]
The terms for $i=\mu p-e$ and $i=\mu p-e+e(p-1)$ contribute the leading term in the variable $\varpi$
\begin{align*} &\frac{\mu p}{e}\cdot\varpi^{\mu p-e}-\left(\frac{\mu p+e(p-1)}{e}-\frac{\mu p}{e}x\right)\cdot\varpi^{\mu p-e+e(p-1)}+\cdots\\
=&\left(-\frac{\mu}{e}-\frac{\mu p+e(p-1)}{e}+\frac{\mu p}{e}x\right)\cdot\varpi^{\mu p-e+e(p-1)}+\cdots\end{align*}
since, similarly to (\ref{twoterms}), we have $p\nmid -\frac{\mu}{e}+1$ as $e<p$. For the terms with $i\neq \mu p-e+e(p-1)$, $\mu p-e$ we must again verify that
\[ \vw(b_i\varpi^i)>\mu p-e+e(p-1). \]
This is clear for $i>\mu p-e+e(p-1)$ and for
\[ \mu p-e<i< \mu p-e+e(p-1)\]
since then $p\mid b_i$. For $i<\mu p-e$ it suffices to show by (\ref{nablaest}) that we have instead $$\vw(a_i\varpi^i)>\mu p+e(p-1)$$ for $i<\mu p$. Since $l_2(a)=\mu p$ we have $\vw(a_i)\geq 2e(p-1)$ for
\[ \mu p-e(p-1)<i<\mu p\]
and hence $\vw(a_i\varpi^i)>\mu p+e(p-1)$. For
\[ \mu p-2e(p-1)<i\leq \mu p-e(p-1)\]
we have by $\vw(a_i)\geq 3e(p-1)$ by Prop. \ref{mainestimate} a) since
\[i\leq \mu p-e(p-1)<-\left(3-\frac{2}{p}\right)\cdot e.\]
Indeed this last inequality is equivalent to
\[\mu p<\left((p-1)-(3-\frac{2}{p})\right)\cdot e\Leftrightarrow \mu < e-\left(\frac{4}{p}-\frac{2}{p^2}\right)\cdot e\]
which holds by our assumption $4e<p$, noting that $e-1$ is the maximal value for $\mu$.
Finally for $$i\leq \mu p-2e(p-1)<-e-e(p-1)=-ep$$
we have by (\ref{growth})
\begin{multline*}  \vw(a_i\varpi^i)\geq -(p-1)i-pe > (p-1)pe-pe=(p-2)pe\geq (2p-3)e\\
=-e+2e(p-1)>\mu p+e(p-1).\end{multline*}

\end{proof}

\begin{prop} If $e<p/4$ the identity (\ref{sigmavol}) holds for $r=2$.
\label{r2prop}\end{prop}

\begin{proof} By Lemmas \ref{notdivbyp} and \ref{divbyp} we can choose $\alpha_i$ such that
$$ \vw(\nabla\alpha_i)=l(\nabla\alpha_i)=\begin{cases} n_i-2e & \text{if $p\nmid n_i$ and $p\nmid n_i-e$}\\ n_i-2e+e(p-1) & \text{if $p\mid n_i$ or $p\mid n_i-e$.}\end{cases}$$
As in the proof of Prop. \ref{r1prop}, for each $0<n_1<e$ there is a unique $n_2=pn_1$ divisible by $p$. Similarly for each $n_h$ with $e<n_h<2e$ (which is unique if it exists) there is a unique $$n_{h+1}-e=p(n_h-e)$$ divisible by $p$. Note here that $n_h\leq 2e-1$ and hence $$n_{h+1}\leq p(e-1)+e<e(p-1)$$ using $2e<p$.
Let
$$q_i:=\nabla\alpha_i^{\sigma^{-1}}(\root e \of {\zeta_p-1})\in K$$
be the basis of $L_2$. We again find that
\begin{align*} &p \cdot q_1, p^{-1}\cdot q_2, \dots, p\cdot q_h,p^{-1}\cdot q_{h+1},\dots, q_{f_\eta}
 &\text{if $n_1<e$ and $e<n_h<2e$}\\
&p \cdot q_1, p^{-1}\cdot q_2, \dots, q_h, q_{h+1},\dots, q_{f_\eta} &\text{if $n_1<e$ and $\not\exists$ $e<n_h<2e$}\\
&q_1, q_2, \dots, p\cdot q_h, p^{-1}\cdot q_{h+1},\dots, q_{f_\eta} &\text{if $\not\exists$  $n_1<e$ and $e<n_h<2e$}\\
&q_1, q_2, \dots, q_h, q_{h+1},\dots, q_{f_\eta} &\text{if $\not\exists$ $n_1<e$ nor $e<n_h<2e$}\end{align*}
is a basis of $\co_{K}^{[n_1-2e]}$ and the statement follows.
 \end{proof}

\end{document}